\documentclass{article} 

\usepackage{amsmath,amsthm,amssymb,amsfonts}

\usepackage{geometry} 
\newtheorem{Theorem}{Theorem}[section]
\theoremstyle{plain}

\newtheorem{Lemma}[Theorem]{Lemma}
\newtheorem{Definition}[Theorem]{Definition}
\newtheorem{Remark}[Theorem]{Remark}
\newtheorem{Corollary}[Theorem]{Corollary}
\newtheorem{Proposition}[Theorem]{Proposition}

\numberwithin{equation}{section}

\title{Branching Random Walk \\ in an inhomogeneous breeding potential}
\author{Sergey Bocharov and Simon C. Harris\\University of Bath}

\begin{document}
\maketitle

\begin{abstract}
We consider a continuous-time branching random walk in the inhomogeneous breeding potential $\beta |\cdot|^p$, where $\beta > 0$, $p \geq 0$. 
We prove that the population almost surely explodes in finite time if $p > 1$ and doesn't explode if $p \leq 1$. In the non-explosive cases, we determine the asymptotic behaviour
of the rightmost particle.
\end{abstract}
\section{Introduction and main results}
We consider a branching system with single particles moving independently according to a continuous-time random walk on 
$\mathbb{Z}$. The random walk makes jumps of size $1$ up or down at constant rate $ \lambda > 0$ in each direction. A particle 
currently at position $y \in \mathbb{Z}$ is independently replaced by two new particles at the parent's 
position at instantaneous rate $\beta |y|^p$, where $\beta > 0$ and $p \geq 0$ are some given constants. 

We denote the set of particles present in the system at time $t$ by $N_t$. If $u \in N_t$ then the position 
of a particle $u$ at time $t$ is $X^u_t$ and its path up to time $t$ is $(X^u_s)_{0 \leq s \leq t}$. 
The law of the branching process started with a single initial particle at $x \in \mathbb{Z}$ is denoted by 
$P^x$ with the corresponding expectation $E^x$ and the natural filtration of the process is denoted by $(\mathcal{F}_t)_{t \geq 0}$.

Let us define the explosion time of the population as
\[
T_{explo} = \sup \{ t : |N_t| < \infty \} \text{.}
\]
We have the following dichotomy for $T_{explo}$ in terms of $p$, the exponent of the breeding potential.
\begin{Theorem}[Explosion criterion]
\label{criticalrw}
For the inhomogeneous BRW started at any $x \in \mathbb{Z}$: \newline
a) If $p \leq 1$  then $T_{explo} = \infty \ P^x$-a.s. \newline
b) If $p > 1$  then $T_{explo} < \infty \ P^x$-a.s.
\end{Theorem}
Let us also define the process of the rightmost particle as 
\[
R_t := \sup_{u \in N_t} X^u_t \ \text{, } \qquad t \geq 0 \text{.}
\]
For $p \in [0, 1]$, we prove the following result about the asymptotic behaviour of $R_t$.
\begin{Theorem}[Rightmost particle asymptotics]
\label{main}
For the inhomogeneous BRW and any $x \in \mathbb{Z}$: \newline
a) If $p = 0$ then
\begin{equation}
\label{pzero}
\lim_{t \to \infty} \frac{R_t}{t} = \lambda( \hat{\theta} - \frac{1}{\hat{\theta}} ) \qquad P^x \text{-a.s.,}
\end{equation}
where $\hat{\theta}$ is the unique solution of
\begin{equation}
\label{thetacritical}
\big( \theta - \frac{1}{\theta} \big) \log \theta - \big( \theta + \frac{1}{\theta} \big) + 2 = \frac{\beta}{\lambda} \quad \text{on } (1, \infty) 
\end{equation}
b) If $p \in (0, 1)$ then
\begin{equation}
\label{pzero_one}
\lim_{t \to \infty} \Big( \frac{\log t}{t} \Big)^{\hat{b}} R_t = \hat{c} 
\qquad P^x \text{-a.s.,}
\end{equation}
where $\hat{b} = \frac{1}{1-p}$ and $\hat{c} = \Big( \frac{\beta(1-p)^2}{p} \Big)^{\hat{b}} $. \newline
c) If $p = 1$ then
\begin{equation}
\label{pone}
\lim_{t \to \infty} \dfrac{\log R_t}{\sqrt{t}} = \sqrt{2 \beta} \qquad P^x \text{-a.s.}
\end{equation}
\end{Theorem}

\noindent Note that Part a) of Theorem \ref{main} is a special case of a result proved by Biggins  \cite{27,26}.

We can compare Theorems \ref{criticalrw} and \ref{main} for this branching random walk in an inhomogenous branching potential with some analogous known results for Branching Brownian Motion.
Consider a model for branching Brownian motion in an inhomogeneous potential where single particles move as standard Brownian motions, each branching into two new particles at instantaneous rate $\beta|x|^p$ when at position $x$, where $\beta > 0$, $p \geq 0$.
This inhomogeneous BBM has been considered in It\^o \&  McKean \cite{4},  Harris \& Harris \cite{3} and Berestycki et al. \cite{25a, 25} where, in particular, we find the following results:

\begin{Theorem}[It\^o \&  McKean \cite{4}, Section 5.14.]
\label{criticalbbm}
Consider a BBM in the potential $\beta|\cdot|^p$, $\beta > 0$, $p \geq 0$ 
started from $x \in \mathbb{R}$: \newline
a) If $p \leq 2$  then $T_{explo} = \infty \ P^x$-a.s. \newline
b) If $p > 2$  then $T_{explo} < \infty \ P^x$-a.s.
\end{Theorem}
\begin{Theorem}[Harris \& Harris \cite{3}]
\label{mainBBM}
Consider the BBM model with $\beta > 0$, $p \in [0, 2]$, $x \in \mathbb{R}$. \newline
a) If  $p \in [0,2)$ then
\begin{equation}
\label{pzero_two}
\lim_{t \to \infty} \frac{R_t}{t^{\hat{b}}} \ = \ \hat{a} \qquad P^x \text{-a.s.}
\end{equation}
where $\hat{b} = \frac{2}{2-p}$ and $\hat{a} = \Big( \frac{\beta}{2}(2 - p)^2 \Big)^{\frac{1}{2 - p}} $. \newline
b) If $p = 2$ then 
\begin{equation}
\label{ptwo}
\lim_{t \to \infty} \dfrac{\log R_t}{t} = \sqrt{2 \beta} \qquad P^x \text{-a.s.}
\end{equation}
\end{Theorem}
Comparing results, it can be seen that the inhomogeneous Branching Random Walk shows quite a different behaviour from the inhomogeneous Branching Brownian Motion, both in terms of the explosion criteria and regarding the asymptotic growth of the rightmost particle position.


We shall give a heuristic argument to help explain Theorems \ref{criticalrw} - \ref{mainBBM} in Section 2. The rest of the paper will then 
contain the detailed proofs of Theorems \ref{criticalrw} and \ref{main}.
In Section 3 we introduce a family of one-particle martingales. We also present 
some other relevant one-particle results, which will be used in later sections. 
Section 3 is self-contained and can be read out of the context of branching 
processes. In Section 4 we recall some standard techniques used in the analysis of branching systems, which include 
spines, additive martingales and martingale changes of measure. 
In Section 5 we prove Theorem \ref{criticalrw} about the explosion time 
using standard spine methods. 
Section 6 is devoted to the proof of Theorem \ref{main} about the rightmost particle 
using the spine methods again.
\section{Heuristics}
Theorems \ref{criticalrw} - \ref{mainBBM} are concerned with \emph{almost sure} explosion and \emph{almost sure} rightmost particle asymptotics. 
We can can informally recover analogous \emph{expectation} results with careful use of the well known Many-to-One Lemma (for example, see \cite{2}), which reduces the expectation of the sum of functionals 
of particles alive at time $t$ to the expectation of a single particle. 

In particular, the expected number of particles alive at time $t$ in the branching system is 
\begin{equation}
\label{eq_Nt}
E^x |N_t| = \mathbb{E}^x e^{\int_0^t \beta(X_s) \mathrm{d}s} =  \mathbb{E}^x e^{\int_0^t \beta|X_s|^p \mathrm{d}s}
\end{equation}
where $(X_t)_{t \geq 0}$ is the single-particle process under $\mathbb{P}^x$. 
It is then relatively straightforward to check that if $(X_t)_{t \geq 0}$ is a Brownian motion, the expected number of particles at time $t$
is: finite for all $t>0$ if $p<2$;  finite for $t<\hat t$ and infinite for $t\geq \hat t$ for some constant $\hat t$ when $p=2$; and, infinite for all $t>0$ if $p>2$. 
Whereas, if $(X_t)_{t \geq 0}$ is a continuous-time random walk then the expected number of particles at time $t$  is: 
finite for all $t>0$ if $p<1$; and,  infinite for all $t>0$ if $p>1$. 
These computations give the critical value of $p$ for explosion of the expected numbers of particles, 
and suggest the almost sure explosion criteria found in Theorems \ref{criticalrw} and \ref{criticalbbm}

The expected number of particles following 'close' to a given trajectory $f$ up to time $t$ is
\begin{equation}
\label{eq_Rt}
E^x \Big( \sum_{u \in N_t} \mathbf{1}_{ \{ X^u_s \approx f(s) \ \forall s \in [0,t] \} } \Big) = \mathbb{E}^x 
\Big( \mathbf{1}_{ \{ X_s \approx f(s) \ \forall s \in [0,t] \} } e^{\int_0^t \beta |X_s|^p \mathrm{d}s} \Big) \text{.}
\end{equation}
If $(X_t)_{t \geq 0}$ is a continuous-time random walk then using heuristic methods 
which involve large deviations theory for L\'evy processes (for example, see \cite{123}), we find  
\[
\log \mathbb{E}^x \Big( \mathbf{1}_{ \{ X_s \approx f(s) \ \forall s \in [0,t] \} } e^{\int_0^t \beta |X_s|^p \mathrm{d}s} \Big) \sim 
I_t(f):= \int_0^t \beta f(s)^p - \Lambda \big( f'(s) \big) \mathrm{d}s \text{,}
\]
where $ \Lambda : [0, \infty) \to [0, \infty)$ is the rate function of the random walk given by 
\[
\Lambda(x) = 2 \lambda + x \log \big( \frac{\sqrt{x^2 + 4 \lambda^2} + x}{2 \lambda} \big) - 
\sqrt{x^2 + 4 \lambda^2} \sim x \log x \text{ as } x \to \infty \text{.}
\]
(See Schilder's theorem for large deviations of paths in Brownian motion, where $\Lambda(x) = \frac{1}{2}x^2$.) Hence the expected number of particles following the curve $f$ either grows 
exponentially or decays exponentially in $t$ depending on the growth rate of 
$f$. 

Further, we anticipate that the almost sure number of particles that have stayed close to path $f$
over large time period $[0,t]$ will be roughly of order $\exp\{I_t(f)\}$ \emph{as long as there have not been any extinction events along the path}, corresponding to the growth rate always remaining positive with $I_s(f)>0$ for all $s\in(0,t]$. 
See Berestycki et al. \cite{25} where such almost sure growth rates along paths are made rigorous for inhomogeneous BBM.

Thus, in order to find the almost sure asymptotic rightmost particle position, for  $t$ large we would like to find $\sup f(t)$ where the supremum is taken over all paths such that no extinction occurs, that is, over paths $f$ with  $I_s(f)>0$ for all $s\in(0,t]$. 
In fact, it turns out that the optimal path $f^*$ for the rightmost position then satisfies $I_s(f^*)=0$ for all $s\in(0,t]$, that is, $f^*$ solves the equation
\[
\Lambda \big( {f^*}'(s) \big) = \beta f^*(s)^p.
\]
Solving this equation for the inhomogenous BRW leads exactly to the asymptotics of the rightmost particle as given in Theorem \ref{main}. 
Although we will not make the above heuristics rigorous for the BRW in this article, our more direct proof of Theorem \ref{main}, which we give in Section 6, will involve showing that 
there almost surely exists a particle staying close to the critical curve $f^*$.

\section{Single-particle results}
In this section we introduce a family of martingales for continuous-time random walks. 
Throughout this section the time set for all the processes is assumed to be $[0, T)$, where 
$T \in (0, \infty]$ is deterministic.

Suppose we are given a Poisson process $(Y_t)_{t \in [0,T)} \stackrel{d}{=} PP(\lambda)$ 
under a probability measure $ \mathbb{P}$. Let us denote by $J_i$ the time of the 
$i^{th}$ jump of $(Y_t)_{t \in [0, T)}$. Then we have the following result. 
\begin{Lemma} 
\label{lem_po_mart}
Let $ \theta : [0, T) \to [0, \infty)$ be a locally-integrable function. 
That is, \newline $ \int_0^t \theta(s) \mathrm{d}s < \infty \ \forall t \in [0, T) $. Then 
the following process is a $\mathbb{P}$-martingale:
\begin{equation*}
M_t := e^{\int_0^t \log \theta(s) \mathrm{d}Y_s + \lambda \int_0^t (1 -
\theta(s)) \mathrm{d}s} = \Big( \prod_{i : J_i \leq t} \theta(J_i) \Big) e^{\lambda \int_0^t (1 -
\theta(s)) \mathrm{d}s} \ \text{, } t \in [0, T) \text{,}
\end{equation*}
where for any function $f$, $\int_0^t f(s) \mathrm{d}Y_s := \sum_{i : J_i \leq t} f(J_i)$. 
\end{Lemma}
The next result tells what effect the martingale $(M_t)_{t \in [0, T) }$ has on the 
process $(Y_t)_{t \in [0, T)}$ when used as a Radon-Nikodym derivative.
\begin{Lemma}
\label{lem_po_meas}
Let $(\hat{\mathcal{F}}_t)_{t \in [0, T)}$ be the natural filtration of $(Y_t)_{t \in [0, T)}$. 
Define the new measure $ \mathbb{Q}$ via 
\begin{equation*}
\frac{\mathrm{d}\mathbb{Q}}{\mathrm{d}\mathbb{P}}
\bigg\vert_{\hat{\mathcal{F}}_t} = M_t \quad \text{, } t \in [0, T) \text{.}
\end{equation*}
Then under the new measure $ \mathbb{Q}$
\[
(Y_t)_{t \in [0, T)} \stackrel{d}{=} IPP \big( \lambda \theta(t) \big) \text{,}
\]
where $IPP \big( \lambda \theta(t) \big)$ stands for time-inhomogeneous 
Poisson process of instantaneous jump rate $ \lambda \theta(t)$.
\end{Lemma}
\begin{proof}[Outline of the proof of Lemmas \ref{lem_po_mart} and  \ref{lem_po_meas}]
As an intermediate step one can check by standard calculations that the following identity holds:
\begin{equation}
\label{eq_312}
\mathbb{E} \Big(e^{\int_0^t \log \theta(s) \mathrm{d}Y_s} \mathbf{1}_{\{Y_t = k\}}
\Big) = e^{-\lambda t} \frac{\lambda^k}{k!} \Big(\int_0^t \theta(s)
\mathrm{d}s \Big)^k \quad \forall k \in \mathbb{N} \text{,}
\end{equation}
where $ \mathbb{E}$ is the expectation associated with $\mathbb{P}$.

The martingale property of $(M_t)_{t \in [0, T)}$ then follows immediately. 

To verify that under $ \mathbb{Q} $, $ (Y_t)_{t \in [0, T)}$ is a time-inhomogeneous Poisson process one 
can check the finite-dimensional distributions.
\end{proof}
For the next few results suppose that $(Y_t)_{t \in [0, T)} \stackrel{d}{=} IPP(r(t))$, where  $r : [0, T) \to [0, \infty)$ is a 
locally-integrable function. That is, $ (Y_t)_{t \in [0, T)}$ is a time-inhomogeneous Poisson process with 
instantaneous jump rate $r(t)$. 

The following identity is a standard integration by-parts-formula which is trivial to prove.
\begin{Proposition}[Integration by parts for time-inhomogeneous Poisson processes]
\label{poisson_b}
Let $f \in C^1 \big( [0, T) \big) $. Then almost surely 
\[
\int_0^t f(s) \mathrm{d} Y_s = f(t) Y_t - \int_0^t f'(s) Y_s \mathrm{d}s \text{,} 
\]
\end{Proposition}
Since $(Y_t)_{t \in [0, T)} \stackrel{d}{=} (Z_{R(t)})_{t \in [0, T)}$, where $R(t) := \int_0^t r(s) ds$ 
and $(Z_t)_{t \geq 0} \stackrel{d}{=} PP(1)$ we also have the following useful result.
\begin{Proposition}[SLLN for time-inhomogeneous Poisson processes]
\label{poisson_c} $ $ \newline
If $\lim_{t \to T} \int_0^t r(s) \mathrm{d}s = \infty \ $ then 
\[
\dfrac{Y_t}{\int_0^t r(s) \mathrm{d}s} \to 1 \text{ a.s.  as } t \to T \text{.}
\]
\end{Proposition}
The next result combines Propositions \ref{poisson_b} and \ref{poisson_c}.
\begin{Proposition}
\label{poisson_d}
Let $f:[0, T) \to [0, \infty) $ be differentiable such that $f'(t) \geq 0$ for $t$ large enough.
Suppose $r$ and $f$ satisfy the following two conditions:
\begin{enumerate}
\item $\int_0^t r(s) \mathrm{d}s \to \infty$ as $t \to T$
\item $\limsup_{t \to T} \frac{f(t) \int_0^t r(s) \mathrm{d}s}{\int_0^t f(s)r(s) 
\mathrm{d}s} < \infty$
\end{enumerate}
Then 
\[
\frac{\int_0^t f(s) \mathrm{d} Y_s}{\int_0^t f(s)r(s) \mathrm{d}s}
\to 1 \text{ a.s. as } t \to T \text{.} 
\]
\end{Proposition}
Note that the second condition is generally rather restrictive, but it will be satisfied
by the functions that we consider in this article.
\begin{proof}
Observe that by Proposition \ref{poisson_b} we have 
\[
\frac{\int_0^t f(s) \mathrm{d} Y_s}{\int_0^t f(s)r(s) \mathrm{d}s} = \frac{f(t) 
Y_t - \int_0^t f'(s) Y_s \mathrm{d}s}{\int_0^t f(s)r(s) \mathrm{d}s} \text{.}
\]
Then apply Proposition \ref{poisson_c} and use the deterministic integration-by-parts formula.
\end{proof}
Let us now consider a continuous-time random walk  $(X_t)_{t \in [0, T)}$ defined under some probability 
measure $ \mathbb{P}$ as it was described in the introduction. It can be written as a difference of two 
independent Poisson processes of rate $\lambda$:
\[
X_t = X^+_t - X^-_t \ \text{, } t \in [0, T) \text{,}
\]
where $(X^+_t)_{t \to [0, T)}$ is the process of positive jumps and 
 $(X^-_t)_{t \in [0, T)} \stackrel{d}{=} PP( \lambda)$ is the process of negative jumps.
From Lemmas \ref{lem_po_mart} and \ref{lem_po_meas} we get the following result.
\begin{Proposition} 
\label{prop_ch_meas}
Let $\theta^+$, $\theta^- : [0, T) \to [0, \infty)$ be two locally-integrable functions. 
Then the following process is a $\mathbb{P}$-martingale:
\begin{equation}
\label{eq_ch_meas}
M_t := e^{\int_0^t \log \theta^+(s) \mathrm{d}X^+_s + \lambda \int_0^t (1 -
\theta^+(s)) \mathrm{d}s \ + \ \int_0^t \log \theta^-(s) \mathrm{d}X^-_s + \lambda \int_0^t (1 -
\theta^-(s)) \mathrm{d}s} \ \text{, } t \in [0, T) \text{.}
\end{equation}
Moreover, if we define the new measure $ \mathbb{Q}$ as
\begin{equation*}
\frac{\mathrm{d}\mathbb{Q}}{\mathrm{d}\mathbb{P}}
\bigg\vert_{\hat{\mathcal{F}}_t} = M_t \quad \text{, } t \in [0, T) \text{,}
\end{equation*}
where $(\hat{\mathcal{F}}_t)_{t \in [0, T)}$ is the natural filtration of $(X_t)_{t \in [0, T)}$, 
then under $ \mathbb{Q}$
\[
(X^+_t)_{t \in [0, T)} \stackrel{d}{=} IPP \big( \lambda \theta^+(t) \big) \text{, } \ 
(X^-_t)_{t \in [0, T)} \stackrel{d}{=} IPP \big( \lambda \theta^-(t) \big) \text{.}
\]
\end{Proposition}
In other words the martingale $M$ used as the Radon-Nikodym derivative has the 
effect of scaling the upward jumps by the factor of $ \theta^+(t)$ and the 
rate of downward jumps by the factor $ \theta^-(t)$ at time $t$. 

Furthermore from Propositions \ref{poisson_c} and \ref{poisson_d} we know that 
$\mathbb{Q}$-a.s. 
\[
\lim_{t \to T} \frac{X_t^+}{\int_0^t \lambda \theta^+(s) \mathrm{d}s} = 1 \text{, } \ 
\lim_{t \to T} \frac{X_t^-}{\int_0^t \lambda \theta^-(s) \mathrm{d}s} = 1 \text{, } 
\]
\[
\lim_{t \to T} \frac{\int_0^t f(s) \mathrm{d}X_s^+}{\int_0^t \lambda \theta^+(s) f(s) \mathrm{d}s} = 1 \text{, } \ 
\lim_{t \to T} \frac{\int_0^t f(s) \mathrm{d}X_s^-}{\int_0^t \lambda \theta^-(s) f(s) \mathrm{d}s} = 1 
\]
provided that $ \theta^+$, $ \theta^-$ and $f$ satisfy the conditions of Propositions \ref{poisson_c} and \ref{poisson_d}.
\section{Spines and additive martingales}
In this section we give a brief overview of the main spine tools. The major 
reference for this section is the work of Hardy and Harris \cite{2} where all the proofs 
and further references can be found.

Firstly, let us take the time set of our model to be $[0, T)$ for some deterministic 
$T \in (0, \infty]$. We assume in this section that the branching process starts from $0$.

We let $(\mathcal{F}_t)_{t \in [0, T)}$ denote the natural filtration of our branching process as described in 
the introduction. We define $\mathcal{F}_T := \sigma(\cup_{t \in [0, T)} \mathcal{F}_t)$.

Let us now extend our branching random walk by identifying an infinite line of descent, which we refer to as the spine, 
in the following way. The initial particle of the branching process begins the spine. When it splits into two new particle, one of them 
is chosen with probability $\frac{1}{2}$ to continue the spine. This goes on in the obvious way: whenever the particle 
currently in the spine splits, one of its children is chosen uniformly at random to continue the spine.

The spine is denoted by $\xi = \{ \varnothing, \xi_1, \xi_2, \cdots\}$, where $\varnothing$ is the initial particle (both in the spine 
and in the entire branching process) and $\xi_n$ is the particle in the $(n+1)^{st}$ generation of the spine. Furthermore, at time $t \in [0, T)$ 
we define: 
\begin{itemize}
\item $node_t(\xi) := u \in N_t \cap \xi$ (such $u$ is unique). That is, $node_t(\xi)$ is the particle in the spine alive at time $t$.
\item $n_t := |node_t(\xi)|$. Thus $n_t$ is the number of fissions that have occured along the spine by time $t$. 
\item $\xi_t := X^u_t$ for $u \in N_t \cap \xi$. So $(\xi_t)_{t \in [0,T)}$ is the path of the spine.
\end{itemize}
The next important step is to define a number of filtrations of our sample 
space, which contain different information about the process.
\begin{Definition}[Filtrations]$ $
\begin{itemize}
\item $ \mathcal{F}_t$ was defined earlier. It is the filtration which 
knows everything about the particles' motion and their genealogy, but it 
knows nothing about the spine.
\item We also define 
$ \tilde{\mathcal{F}}_t := \sigma \big( \mathcal{F}_t, node_t(\xi) \big) $. 
Thus $ \tilde{\mathcal{F}}$ has all the information about the process 
together with all the information about the spine. This will be the largest 
filtration.
\item $ \mathcal{G}_t := \sigma \big( \xi_s : 0 \leq s \leq t \big)$. 
This filtration only has information about the path of the spine process, 
but it can't tell which particle $u \in N_t$ is the spine particle at time $t$.
\item $ \tilde{\mathcal{G}}_t := \sigma \big( \mathcal{G}_t , \ ( node_s(\xi) : 
0 \leq s \leq t) \big)$. This filtration knows everything about the spine 
including which particles make up the spine, but it 
doesn't know what is happening off the spine.
\end{itemize}
\end{Definition}
Note that $ \mathcal{G}_t \subset \tilde{\mathcal{G}}_t 
\subset \tilde{\mathcal{F}}_t$ and $\mathcal{F}_t \subset \tilde{\mathcal{F}}_t$. 
We shall be using these filtrations throughout the whole article for taking various conditional 
expectations. 

We let $\tilde{P}$ be the probability measure under which the branching random walk is defined together with the 
spine. Hence $P = \tilde{P} \vert_{\mathcal{F}_T}$. We shall write $\tilde{E}$ for 
the expectation with respect to $ \tilde{P}$.

 Under $\tilde{P}$ the entire branching process (with the spine) can be described in the following way.
\begin{itemize}
\item the initial particle (the spine) moves like a random walk.
\item At instantaneous rate $ \beta |\cdot|^p$ it splits into two new particles. 
\item One of these particles (chosen uniformly at random) continues the spine.
That is, it continues moving as a random walk and branching at rate 
$ \beta |\cdot|^p$. 
\item The other particle initiates a new independent $P$-branching processes 
from the position of the split
\end{itemize}
It is not hard to see that under $\tilde{P}$ the spine's path $(\xi_t)_{t \in [0, T)}$ is itself a 
continuous-time random walk. 

Also, conditional on the path of the spine, $(n_t)_{t \in [0, T)}$ is a 
time-inhomogeneous Poisson process (or a Cox process) with instantaneous jump rate $ \beta| \xi_t|^p$. 
That is, conditional on $ \mathcal{G}_t$, $k$ splits take place along the spine by time $t$
with probability 
\[
\tilde{P}(n_t = k | \mathcal{G}_t) = 
\frac{(\int_0^t \beta| \xi_s|^p \mathrm{d}s )^k}{k!} e^{- \int_0^t \beta|\xi_s|^p \mathrm{d}s }
\text{.}
\]
The next result (see e.g. \cite{2}) has already been mentioned in the introduction.
\begin{Theorem}[Many-to-One Theorem]
\label{manytoone}
Let $f(t) \in m \mathcal{G}_t$. In other words, $f(t)$ is 
$\mathcal{G}_t$-measurable. Suppose it has the representation 
\[
f(t) = \sum_{u \in N_t} f_u(t) \mathbf{1}_{ \{ node_t( \xi) = u \} } \text{,} 
\]
where $f_u(t) \in m \mathcal{F}_t$, then
\[
E \Big( \sum_{u \in N_t} f_u(t) \Big) = \tilde{E} \Big( f(t) e^{\int_0^t \beta( \xi_s) \mathrm{d}s} \Big) \text{.}
\]
\end{Theorem}
Now let $ \theta = (\theta^+, \theta^-)$, where $\theta^+$, $\theta^- : [0, T) \to [0, \infty)$ 
are two locally-integrable functions. In view of Proposition \ref{prop_ch_meas} we define the following 
$ \tilde{P}$-martingale w.r.t filtration $(\tilde{\mathcal{G}}_t)_{t \in [0, T)}$:
\begin{align}
\label{M_tilde1_}
\tilde{M}_{\theta}(t) & := e^{- \beta \int_0^t \vert \xi_s \vert^p \mathrm{d}s}
2^{n_t} \times \exp \Big( \int_0^t \log \theta^+(s) \mathrm{d} \xi_s^+ +
\int_0^t \lambda (1 - \theta^+(s)) \mathrm{d}s \nonumber \\
& \qquad + \int_0^t \log \theta^-(s) \mathrm{d} \xi_s^- + \int_0^t \lambda (1 -
\theta^-(s)) \mathrm{d}s  \Big) \text{,}
\end{align}
where $(\xi^+_t)_{t \in [0, T)}$ is the process of positive jumps of the spine process and 
$(\xi^-_t)_{t \in [0, T)}$ is the process of its negative jumps.

Note that $\tilde{M}_{\theta}$ is the product of two $ \tilde{P}$-martingales, the first of which 
doubles the branching rate along the spine, and the second biases the rates of upward and downward 
jumps of the spine process. If we define the probability measure $ \tilde{Q}_{\theta}$ as
\begin{equation}
\label{Q_tilde1_}
\dfrac{\mathrm{d} \tilde{Q}_{\theta}}{\mathrm{d} \tilde{P}} 
\bigg\vert_{\tilde{\mathcal{F}}_t} = \tilde{M}_{\theta}(t) \ 
\text{, } \qquad t \in [0, T)
\end{equation}
then under $\tilde{Q}_{\theta}$ the branching process has the following description: 
\begin{Proposition}[Branching process under $\tilde{Q}_{\theta}$]
\label{BRW_under_Q} $ $
\begin{itemize}
\item The initial particle (the spine) moves like a biased random walk. That is, at time $t$ it jumps up at 
instantaneous rate $ \lambda \theta^+(t)$ and jumps down at instantaneous rate $ \lambda \theta^-(t)$. 
\item When it is at position $x$ it splits into two new particles at instantaneous rate $ 2 \beta |x|^p$.
\item One of these particles (chosen uniformly at random) continues the spine. I.e. it 
continues moving as a biased random walk and branching at rate $2 \beta | \cdot|^p$.
\item The other particle initiates an unbiased branching process (as under $P$) from the position of 
the split.
\end{itemize}
\end{Proposition}
Note that although \eqref{Q_tilde1_} only defines $ \tilde{Q}_{\theta}$ on events in 
$ \cup_{t \in [0, T)} \tilde{\mathcal{F}}_t$, Carath\'eodory's extension theorem tells that $ \tilde{Q}_{\theta}$ has a 
unique extension on $ \tilde{\mathcal{F}}_T := \sigma( \cup_{t \in [0, T)} \tilde{\mathcal{F}}_t)$ 
and thus \eqref{Q_tilde1_} implicitly defines $ \tilde{Q}_{\theta}$ on $ \tilde{\mathcal{F}}_T$.
\begin{Proposition}[Additive martingale]
\label{Q_theta}
We define the probability measure $Q_{\theta}$ := $\tilde{Q}_{\theta}|_{\mathcal{F}_{T}}$ so that 
\begin{equation}
\label{Q_1_}
\dfrac{\mathrm{d} Q_{\theta}}{\mathrm{d} P} \bigg\vert_{\mathcal{F}_t} = M_{\theta}(t) \ 
\text{, } \qquad t \in [0, T) \text{,}
\end{equation}
where $M_{\theta}(t)$ is the additive martingale
\begin{align}
\label{additivemart1_}
M_{\theta}(t) & =  \sum_{u \in N_t} \exp \Big( \int_0^t \log \theta^+(s)
\mathrm{d}X^+_u(s) + \int_0^t \log \theta^-(s) \mathrm{d}X^-_u(s) \nonumber \\
& \qquad + \int_0^t \lambda \big( 2 - \theta^+(s) - \theta^-(s) \big) \mathrm{d}s - \beta
\int_0^t \vert X_u(s) \vert^p \mathrm{d}s
\Big)
\end{align}
and $(X^+_u(s))_{0 \leq s \leq t}$ is the process of positive jumps of particle $u$, 
$(X^-_u(s))_{0 \leq s \leq t}$ is the process of its negative jumps. 
\end{Proposition}

Let us recall the following measure-theoretic result, which gives Lebesgue's decomposition 
of $Q_{\theta}$ into absolutely-continuous and singular parts. It can for example be found 
in the book of R. Durrett \cite{6} (Section 4.3).
\begin{Lemma}
\label{durrett_lem2_}
For events $A \in \mathcal{F}_T$
\begin{equation}
\label{measuredecomposition2_}
Q_{\theta} \big( A \big) = \int_A \limsup_{t \to T} M_{\theta}(t) \mathrm{d}P + 
Q_{\theta} \big( A \cap \{ \limsup_{t \to T} M_{\theta}(t) = \infty \} \big) \text{.}
\end{equation}
\end{Lemma}
In view of this lemma one will be interested in identifying the set of values of $\theta$ 
for which $ \limsup_{t \to T} M_{\theta}(t) < \infty \ Q_{\theta}$-a.s., in which case 
$Q_{\theta} \ll P$ on $ \mathcal{F}_T$. An important tool for doing this is the so-called 
spine decomposition.
\begin{Lemma}[Spine decomposition]
\label{spine_decomposition}
\begin{align}
\label{spine_dec}
E^{\tilde{Q}_{\theta}} \Big( M_{\theta}(t) \big\vert \tilde{\mathcal{G}}_T \Big) &= 
\exp \Big( \int_0^t \log \theta^+(s) \mathrm{d} \xi^+_s + \int_0^t \log \theta^-(s) \mathrm{d} \xi^-_s \nonumber \\
&\qquad \qquad + \lambda \int_0^t (2 - \theta^+(s) - \theta^-(s)) \mathrm{d}s - \beta \int_0^t \vert \xi_s \vert^p \mathrm{d}s \Big) \nonumber \\
+ &\sum_{u < node_t(\xi)} \exp \Big( \int_0^{S_u} \log \theta^+(s) \mathrm{d} \xi^+_s + \int_0^{S_u} \log \theta^-(s) \mathrm{d} \xi^-_s \nonumber \\
&\qquad \qquad + \lambda \int_0^{S_u} (2 - \theta^+(s) - \theta^-(s)) \mathrm{d}s - \beta \int_0^{S_u} \vert \xi_s \vert^p \mathrm{d}s \Big) \text{,}
\end{align}
where $ \{  S_u : u \in \xi \} $ is the set of fission times along the spine. 

The first term is called the spine term or $spine(t)$ and the second one is called the sum term or $sum(t)$.
\end{Lemma}
\section{Explosion: proof of Theorem \ref{criticalrw}}
\subsection{Case $p \leq 1$}
Firstly, we shall prove that $T_{explo} = \infty \ P^x$-a.s. if the exponent of the branching rate $p$ is $ \leq 1$. 
As in the proof of Theorem \ref{criticalbbm} a) from \cite{4} for the BBM model it will be sufficient to show that $E|N_t| < \infty$ 
for some $t > 0$ as it is explained below.

Let us begin with the simple observation, which says that the starting position of the branching process is not important in 
Theorem \ref{criticalrw}. Thus we shall take it to be $0$ in the rest of this section.
\begin{Proposition}
\label{xy}
\[
P^x \big( T_{explo} = \infty \big) = P^y \big( T_{explo} = \infty \big) 
\qquad \forall x,y \in \mathbb{Z} \text{.}
\]
\end{Proposition}
\begin{proof}
Take any $x$ and $y \in \mathbb{Z}$ and start a branching random walk from $x$. 
Let $T_y$ be the first passage time of the process to level $y$. That is, 
\[
T_y := \inf \{ t \ : \ \exists u \in N_t \text{ s.t. } X^u_t = y \} \text{.}
\]
$T_y < \infty$ because a random walk started from any level $x$ will hit any 
level $y$. Then by the strong Markov property of the branching process the subtree initiated from 
$y$ at time $T_y$ has the same law as a branching random walk started from $y$. 
Consequently, if the explosion does not happen in the big tree started from $x$, it cannot 
happen in its subtree started from $y$. Thus
\[
P^x \big( T_{explo} = \infty \big) \leq P^y \big( T_{explo} = \infty \big) \text{.}
\] 
Since $x$ and $y$ were arbitrary it follows that
\[
P^x \big( T_{explo} = \infty \big) = P^y \big( T_{explo} = \infty \big) 
\qquad \forall x,y \in \mathbb{Z} \text{.}
\]
\end{proof}
One important corollary to the previous result is the following 0-1 law. 
\begin{Corollary}
\label{zeroone}
\[
P \big( T_{explo} = \infty \big) \in \{ 0, 1 \} \text{.}
\]
\end{Corollary}
\begin{proof}
If $X_1$ is the position of the first split then from the branching property we have 
\[
P \big( T_{explo} = \infty \big) = E \Big( \big( P^{X_1} (T_{explo} = \infty)
\big)^2 \Big) = \big( P \big( T_{explo} = \infty \big) \big)^2 \text{.}
\]
Thus $ P ( T_{explo} = \infty ) \in \{ 0, 1 \}$.
\end{proof}
Let us now state another useful fact.
\begin{Proposition}
\label{propexplo}
Take some deterministic time $t > 0$.
\[
\text{If} \ P \big( T_{explo} < t \big) = 0
\text{ then} \ P^x \big( T_{explo} < t \big) = 0 \ \forall x \in \mathbb{Z} \text{.}
\]
\end{Proposition}
\begin{proof}
Consider a branching process started from $0$. Take any $ \epsilon \in (0, t)$. Let $T_x$ be the hitting time of level $x$ as in Proposition \ref{xy}.
Then there is a positive probability that the process will hit level $x$ before time $\epsilon$. Then
\begin{align*}
0 &= P \big( T_{explo} < t \big) \geq P \big( T_{explo} < t , T_x < \epsilon \big) \geq 
P \big( T^x_{explo} < t - \epsilon , T_x < \epsilon \big) \\
&= E \Big( P \big( T^x_{explo} < t - \epsilon , T_x < \epsilon | T_x\big)\Big) 
= P \big( T_x < \epsilon \big) P^x \big( T_{explo} < t - \epsilon \big) \text{,}
\end{align*}
where $T^x_{explo}$ is the explosion time of the subtree started from $x$. 
Thus, since $P \big( T_x < \epsilon \big) > 0$ we find that
\[
P^x \big( T_{explo} < t - \epsilon \big) = 0 \text{.}
\]
Since $\epsilon$ was arbitrary, letting $ \epsilon \downarrow 0$ gives the result.
\end{proof}
As a consequence of Proposition \ref{propexplo} we get the following corollary.
\begin{Corollary}
\label{texplo}
Let $t > 0$ be any deterministic time. 
\[
\text{if } P \big( T_{explo} \geq t \big) = 1 \text{ then } P \big( T_{explo} = \infty \big) = 1 \text{.}
\]
In particular, if $E|N_t| < \infty$ then $P(T_{explo} < \infty) = 0$.
\end{Corollary}
\begin{proof}
The result follows by induction since if the original tree almost surely does not explode by time 
$t$ then none of its subtrees initiated at time $t$ will explode by time $2 t$ and one can repeat this 
argument any number of times. 
\end{proof}
\begin{proof}[Proof of Theorem \ref{criticalrw} a)]
We wish to show that if $p \leq 1$ then $P(T_{explo} = \infty) = 1$. From Corollary \ref{texplo}, 
it is sufficient to show that $E(|N_t|) < \infty$ for some $t > 0$. 

By the Many-to-One Theorem (Theorem \ref{manytoone})
\[
E \Big( \vert N_t \vert \Big) = E \Big( \sum_{u \in N_t} 1 \Big) = 
\tilde{E} \Big( e^{\int_0^t \beta \vert \xi_s \vert^p \mathrm{d}s} \Big) \text{,}
\]
where $(\xi_t)_{t \geq 0}$ is a continouos-time random walk under $ \tilde{P}$. Recall, 
$\xi_t = \xi_t^+ - \xi_t^-$, where $(\xi_t^+)_{t \geq 0}$ and $(\xi_t^-)_{t \geq 0}$ 
are two independent Poisson processes with jump rate $\lambda$. Then 
\begin{align*}
\tilde{E} \Big( e^{\int_0^t \beta \vert \xi_s \vert^p \mathrm{d}s} \Big) 
& \leq \tilde{E} \Big( e^{t \beta \sup_{0 \leq s \leq t} |\xi_s|^p} \Big) \\
& = \tilde{E} \Big( e^{t \beta \sup_{0 \leq s \leq t} |\xi_s^+ - \xi_s^-|^p} \Big) 
\leq \tilde{E} \Big( e^{t \beta \sup_{0 \leq s \leq t} \big( (\xi_s^+)^p \vee (\xi_s^-)^p \big) } \Big) \\
& = \tilde{E} \Big( e^{t \beta \big( (\xi_t^+)^p \vee (\xi_t^-)^p \big) } \Big) 
\leq \tilde{E} \Big( e^{t \beta \big( (\xi_t^+)^p + (\xi_t^-)^p \big) } \Big) \\
& = \Big[ \tilde{E} \Big( e^{t \beta (\xi_t^+)^p} \Big) \Big]^2 
\leq \Big[ \tilde{E} \Big( e^{t \beta \xi_t^+ } \Big) \Big]^2
\end{align*}
since $\xi^+$ is supported on $ \{ 0, 1, 2, ...\} $ whence $(\xi_t^+)^p \leq \xi_t^+$ for $p \in [0,1]$. Then 
\[
\tilde{E} \Big( e^{t \beta \xi_t^+ } \Big) = \sum_{n=0}^{\infty} e^{\beta t n} \dfrac{(\lambda t)^n}{n!} e^{- \lambda t}
= \exp \big\{ e^{ \beta t} \lambda t - \lambda t \big\} < \infty \quad \forall t \geq 0 \text{.}
\]
Thus $E(|N_t|) < \infty$ for all $t > 0$.
\end{proof}
\subsection{Case $p > 1$} 
\begin{proof}[Proof of Theorem \ref{criticalrw} b)]
We wish to show that if $p > 1$ then $P(T_{explo} < \infty) = 1$. By Corollary \ref{zeroone} this is equivalent to 
$P(T_{explo} < \infty) > 0$. It would be sufficient to prove that $P(T_{explo} \leq T) > 0$ for all $T > 0$. 
For a contradiction we suppose that there exists $T > 0$ s.t.
\begin{equation}
\label{pexplosion3}
P(T_{explo} \leq T) = 0 \text{.}
\end{equation}
We fix this $T$ for the rest of this subsection. Under the assumption \eqref{pexplosion3} that there is no explosion 
before time $T$ we can perfom the usual spine construction on $[0, T)$. The key steps of the proof can then be summarised as follows:
\begin{enumerate}
\item We choose $\theta^+$, $\theta^- : [0, T) \to [0, \infty)$ such that at time $T$ 
\begin{description}
\item[(A)] the spine process $\xi_t$ goes to $ \infty$ under $\tilde{Q}_{\theta}$
\item[(B)] the additive martingale $M_{\theta}$ from satisfies $\limsup_{t \to T} M_{\theta}(t) < \infty \ Q_{\theta}$-a.s.
\end{description}
\item We deduce that $Q_{\theta} \ll P$ on $\mathcal{F}_T$, whence with positive 
$P$-probability one particle goes to $\infty$ at time $T$ giving infinitely many births 
along its path.
\item We get a contradiction to \eqref{pexplosion3}.
\end{enumerate}
We take $ \theta^-(\cdot) \equiv 1$. That is, we leave the negative jumps of the spine process 
unaltered under $ \tilde{Q}_{\theta}$. $ \theta^+(\cdot)$ needs to be chosen carefully such that both (A) and (B) above 
are satisfied. One such choice is 
\begin{equation}
\label{eq_llll}
\theta^+(s) = (T - s)^{-c} \text{ , } s \in [0, T) \text{,}
\end{equation}
where $c > \frac{p}{p-1}$ (e.g. take $c = \frac{p}{p-1} + 1$).

The additive martingale \eqref{additivemart1_} in this case takes the following form (with $\theta^+(\cdot)$ defined above)
\begin{align}
\label{Qexplo}
M_{\theta}(t) = \sum_{u \in N_t} \exp \Big( \int_0^t \log \theta^+(s) \mathrm{d}X^+_u(s) &+ 
\int_0^t \lambda \big( 1 - \theta^+(s) \big) \mathrm{d}s \nonumber \\
&- \beta \int_0^t \vert X_u(s) \vert^p \mathrm{d}s \Big) \text{ , } \ t \in [0, T) \text{,}
\end{align}
If we can now show that 
\begin{equation}
\label{eq099}
\limsup_{t \to T} M_{\theta}(t) < \infty \qquad Q_{\theta} \text{-a.s.} 
\end{equation}
it would follow from Lemma \ref{durrett_lem2_} that $Q_{\theta} \ll P$ on $\mathcal{F}_T$. 

To prove \eqref{eq099} it is sufficient to show that 
\begin{equation}
\label{eq098}
\limsup_{t \to T} E^{\tilde{Q}_{\theta}} \Big( M_{\theta}(t) \big\vert \tilde{\mathcal{G}}_T \Big) < \infty 
\qquad \tilde{Q}_{\theta} \text{-a.s.,}
\end{equation}
since if \eqref{eq098} holds then by Fatou's lemma
\begin{align*}
E^{\tilde{Q}_{\theta}} \Big( \liminf_{t \to T} M_{\theta}(t) \big\vert \tilde{\mathcal{G}}_T \Big) &\leq
\liminf_{t \to T} E^{\tilde{Q}_{\theta}} \Big( M_{\theta}(t) \big\vert \tilde{\mathcal{G}}_T \Big) \\
&\leq \limsup_{t \to T} E^{\tilde{Q}_{\theta}} \Big( M_{\theta}(t) \big\vert \tilde{\mathcal{G}}_T \Big)
< \infty \quad \tilde{Q}_{\theta} \text{-a.s.,}
\end{align*}
therefore $\liminf_{t \to T} M_{\theta}(t) < \infty \ \tilde{Q}_{\theta}$-a.s. and hence also $Q_{\theta}$-a.s.
Then since $ \frac{1}{M_{\theta}(t)}$ is a positive $Q_{\theta}$-supermartingale on $[0, T)$, it must converge
$Q_{\theta}$-a.s., hence
\[
\limsup_{t \to T} M_{\theta}(t) = \liminf_{t \to T} M_{\theta}(t) < \infty \qquad Q_{\theta} \text{-a.s.} 
\]
So let us now prove \eqref{eq098}. Recall the spine decomposition \eqref{spine_dec}:
\[
E^{\tilde{Q}_{\theta}} \Big( M_{\theta}(t) \big\vert \tilde{\mathcal{G}}_T \Big) = spine(t) + sum(t) \text{,}
\]
where
\[
spine(t) = \exp \Big( \int_0^t \log \theta^+(s) \mathrm{d} \xi^+_s + \int_0^t
\lambda \big( 1 - \theta^+(s) \big) \mathrm{d}s - \int_0^t \beta | \xi_s|^p \mathrm{d}s \Big)
\]
and
\[
sum(t) = \sum_{u < node_t(\xi)} spine(S_u) \text{.}
\]
We start by proving the following assertion about the spine term. 
\begin{Proposition}
\label{spine_bound2}
There exist some $ \tilde{Q}_{\theta}$-a.s. finite positive random variables $C'$, $C''$ and a random 
time $T' \in [0, T)$ such that $ \forall t > T'$
\[
spine(t) \leq C' \exp \Big( - C'' (T - t)^{-p(c-1) + 1} \Big) \text{.}
\]
\end{Proposition}
\begin{proof}[Proof of Proposition \ref{spine_bound2}] 
From Proposition \ref{prop_ch_meas} under $ \tilde{Q}_{\theta}$ the process $(\xi^+_t)_{t \in [0, T)}$ is a time-inhomogeneous Poisson process of 
rate $ \lambda \theta^+(t)$ and $(\xi^-_t)_{t \in [0, T)}$ is a Poisson process of rate $ \lambda$. 

Using the standard integration-by-parts formula one can check that 
\[
\int_0^t (T - s)^{- c} \log (T - s) \mathrm{d}s \sim \frac{1}{c - 1} (T - t)^{- c + 1} \log (T - t) \text{ as } t \to T \text{.}
\]
Hence for  $ \theta^+$ defined as in \eqref{eq_llll} 
\[
\limsup_{t \to T} \dfrac{\log \theta^+(t) \int_0^t \lambda \theta^+(s) \mathrm{d}s}{\int_0^t \log \theta^+(s) \ 
\lambda \theta^+(s) \mathrm{d}s} = 1 \text{.}
\]
Also $\int_0^t \lambda \theta^+(s) \mathrm{d}s = \lambda (c - 1)^{-1} (T - t)^{-c + 1} \to \infty$ as $t \to T$ and 
$\log \theta^+(\cdot)$ is increasing. Thus from Proposition \ref{poisson_c} and Proposition \ref{poisson_d} we have that $\tilde{Q}_{\theta}$ -a.s.
\[
\frac{\xi_t}{\int_0^t \lambda \theta^+(s) \mathrm{d}s} \to 1 \text{ , } \
 \frac{\int_0^t \log \theta^+(s) \mathrm{d} \xi_s^+}{\int_0^t \log \theta^+(s)
\ \lambda \theta^+(s) \mathrm{d}s} \to 1 \text{.}
\]
Combining these observations we get that $ \forall \epsilon > 0 \ \exists \ \tilde{Q}_{\theta}$-a.s. 
finite time $T_{ \epsilon}$ such that $ \forall t > T_{ \epsilon}$ the following inequalities are true:
\begin{align*}
\int_0^t \log \theta^+(s) \mathrm{d} \xi_s^+ &< 
(1 + \epsilon) \int_0^t \log \theta^+(s) \ \lambda \theta^+(s) \mathrm{d}s \\
&= (1 + \epsilon) \int_0^t -c \log(T - s) \lambda(T - s)^{-c} \mathrm{d}s \text{;} 
\end{align*}
\begin{align*}
&|\xi_t| > (1 - \epsilon) \int_0^t \lambda \theta^+(s) \mathrm{d}s = (1 - \epsilon) \frac{\lambda}{c - 1} (T - t)^{-c + 1} \text{;} \\
&\lambda \big( 1 - \theta^+(t) \big) < 0 \text{;} \\
&\log(T - t) \lambda(T - t)^{-c} \leq \dfrac{1}{2} 
\frac{\beta(\frac{\lambda}{c-1}(1-\epsilon))^p}{\lambda c(1 + \epsilon)} (T - t)^{-(c - 1)p} \text{.}
\end{align*}
Thus, for $t > T_{ \epsilon}$ we have 
\begin{align*}
spine(t) &= \exp \Big( \int_0^t \log \theta^+(s) \mathrm{d} \xi^+_s + \int_0^t
\lambda \big( 1 - \theta^+(s) \big) \mathrm{d}s - \int_0^t \beta | \xi_s|^p \mathrm{d}s \Big) \\
&\leq C_{ \epsilon} \exp \Big\{ (1 + \epsilon) \int_0^t -c \lambda \log (T - s) (T - s)^{-c} \mathrm{d}s  \\
& \qquad \qquad \qquad - \beta \int_0^t \Big( \frac{\lambda (1 - \epsilon)}{c - 1} (T - s)^{-c + 1} \Big)^p \mathrm{d}s \Big\} \\
&\leq  C_{\epsilon}' \exp \Big\{ - \frac{1}{2} \beta \Big( \frac{\lambda (1 - \epsilon)}{c - 1} \Big)^p 
\frac{1}{p(c - 1) - 1} (T - t)^{-(c - 1)p + 1} \Big\} \text{,}
\end{align*}
where $C_{ \epsilon}$ and $C_{ \epsilon}'$ are some $\tilde{Q}_{\theta}$-a.s. finite random variables, which don't depend on $t$. 
Letting $T' = T_{\epsilon}$, $C' = C'_{\epsilon}$ and 
$C'' = \frac{1}{2} \beta \Big( \frac{\lambda (1 - \epsilon)}{c - 1} \Big)^p \frac{1}{p(c - 1) - 1}$ we finish 
the proof of Proposition \ref{spine_bound2}. 
\end{proof}
We now look at the $sum$ term: 
\begin{align*}
sum(t) & = \sum_{u < node_t(\xi)} spine(S_u) \\
& = \Big( \sum_{u < node_t(\xi), \ S_u \leq T'} spine(S_u) \Big) + 
\Big( \sum_{u < node_t(\xi), \ S_u > T'} spine(S_u) \Big) \\
& \leq \sum_{u < node_t(\xi), \ S_u \leq T'} spine(S_u) \\
& \qquad + \sum_{u < node_t(\xi), \ S_u > T'} C' \exp \Big( - C'' (T - S_u)^{-p(c-1) + 1} \Big)
\end{align*}
using Proposition \ref{spine_bound2}. The first sum is $ \tilde{Q}_{\theta}$-a.s. bounded 
since it only counts births up to time $T'$. Call an upper bound on the first sum $C_1$. 
Then we have 
\begin{equation}
\label{eq_www2}
sum(t) \leq \ C_1 + C' \sum_{n=1}^{\infty} \exp \Big( - C'' (T - S_n)^{-p(c-1) + 1} \Big) \text{,}
\end{equation}
where $S_n$ is the time of the $n^{th}$ birth on the spine.

The birth process along the spine $(n_t)_{t \in [0, T)}$ conditional on the path of the spine 
is time-inhomogeneous Poisson process (or Cox process) with birth rate $2 \beta |\xi_t|^p$ at 
time $t$. Thus as $t \to T$, almost surely under $\tilde{Q}_{\theta}$
\begin{equation}
\label{eq_deq}
n_t \sim \int_0^t 2 \beta |\xi_s|^p \mathrm{d}s \sim 2 \beta \Big( \frac{\lambda}{c - 1} \Big)^p 
\frac{1}{p(c-1) - 1} (T - t)^{-p(c-1) + 1} \text{,}
\end{equation}
hence, 
\[
n \sim 2 \beta \Big( \frac{\lambda}{c - 1} \Big)^p \frac{1}{p(c-1) - 1} (T - S_n)^{-p(c-1) + 1} \text{.}
\]
So for some $ \tilde{Q}_{\theta}$-a.s. finite positive random variable $C_2$ we have  
\[
(T - S_n)^{-p(c-1) + 1} \geq C_2 n \quad \forall n \text{.}
\]
Then substituting this into \eqref{eq_www2} we get 
\[
sum(t) \leq C_1 + C' \sum_{n = 1}^{ \infty} e^{-C'' C_2 n} \text{,} 
\]
which is bounded $ \tilde{Q}_{\theta}$-a.s. We have thus shown that 
\[
\limsup_{t \to T} E^{\tilde{Q}_{\theta}} \Big( M_{\theta}(t) \big\vert \tilde{\mathcal{G}}_T \Big) =
\limsup_{t \to T}  \Big( spine(t) + sum(t) \Big) < \infty \qquad \tilde{Q}_{\theta} \text{-a.s.}
\]
proving \eqref{eq098} and consequently \eqref{eq099}.

From Lemma \ref{durrett_lem2_} it now follows that for events $A \in \mathcal{F}_T$
\[
Q_{\theta}(A) = \int_A \limsup_{t \to T} M_{\theta}(t) \mathrm{d}P \text{.}
\]
Thus $Q_{\theta}(A) > 0 \Rightarrow P(A) > 0$. Let us consider the event 
$ \big\{ |N_t| \to \infty \text{ as } t \to T \big\} $.
From \eqref{eq_deq} we have $\tilde{Q}_{\theta} \big( n_t \to \infty \text{ as } t \to T \big) = 1$, so 
$Q_{\theta} \big( |N_t| \to \infty \text{ as } t \to T \big) = 1$ and then $P \big( |N_t| \to \infty \text{ as } t \to T \big) > 0$. 
Thus $ P \big( T_{explo} \leq T \big) > 0$, which contradicts the initial assumption \eqref{pexplosion3}. Therefore, 
$P(T_{explo} \leq T) > 0$, $ \forall T > 0$ and hence by Corollary \ref{zeroone}
\[
T_{explo} < \infty \ P \text{-a.s.}
\]
This completes the proof of Theorem \ref{criticalrw}
\end{proof}
\section{The rightmost particle: proof of Theorem \ref{main}}
In this section we consider a branching random walk in the potential $\beta|\cdot|^p$, $\beta > 0$, $p \in [0, 1]$. 
By Theorem \ref{criticalrw} there is no explosion of the population and so we take the time set of the branching process 
to be $[0, \infty)$. That is, in the set-up presented in Section 4 we let $T = \infty$.

Just like with the explosion probability in Section 5, the starting position of the branching process does not affect the behaviour of the 
rightmost particle in Theorem \ref{main}. For example in part a) suppose we know that 
$P^x(\lim_{t \to \infty} t^{-1}R_t = \lambda (\hat{\theta} - \hat{\theta}^{-1}) ) = 1$ for some $x \in \mathbb{Z}$. 
Take some $y \in \mathbb{Z}$. Then a branching process started from $x$ will contain a subtree started from $y$. Hence 
$P^y(\limsup_{t \to \infty} t^{-1}R_t \leq \lambda (\hat{\theta} - \hat{\theta}^{-1}) ) = 1$. Also a branching process started from 
$y$ will contain a subtree started from $x$. Hence $P^y(\liminf_{t \to \infty} t^{-1}R_t \geq \lambda (\hat{\theta} - \hat{\theta}^{-1}) ) = 1$ 
and so $P^y(\lim_{t \to \infty} t^{-1}R_t = \lambda (\hat{\theta} - \hat{\theta}^{-1}) ) = 1$. 
We shall thus take the starting position of the branching process to be $0$ in the forthcoming proof presented in Subsections 6.1 - 6.3. 

Our proof  follows a similar approach as was used for the BBM model in J. Harris and S. Harris in \cite{3}. 
\subsection{Convergence properties of $M_{\theta}$ (under $Q_{\theta}$)}
We let $M_{\theta}$ be the additive martingale as defined in \eqref{additivemart1_} for a given 
parameter $\theta$. Note that since each $M_{ \theta}$ is a positive 
$P$-martingale it must converge $P$-almost surely to a finite limit $M_{\theta}(\infty)$. We are interested in those 
values of $\theta$ for which $M_{\theta}(\infty)$ is strictly positive. The following result deals with this question.
\begin{Theorem}
\label{M_limit} $ $ \newline
\noindent $\underline{ \mathbf{ Case \ A } \ (p = 0) \text{, homogeneous branching} }$: \newline
Recall $ \hat{\theta}$ from \eqref{thetacritical} which solves (uniquely) 
\[
\big( \theta - \frac{1}{\theta} \big) \log \theta - \big( \theta + \frac{1}{\theta} \big) + 2 = \frac{\beta}{\lambda} \quad \text{on } (1, \infty) 
\]
\noindent Consider $ \theta = (\theta^+, \theta^-) $, where $ \theta^+(\cdot) \equiv \theta_0$ 
and $\theta^-(\cdot) \equiv \frac{1}{\theta_0}$ for some constant $\theta_0 > 1$. Then \newline
i) $\theta_0 < \hat{\theta} \ \Rightarrow \ M_{\theta}$ is UI and
$M_{\theta}(\infty) > 0$ a.s. (under $P$). \newline
ii) $\theta_0 > \hat{\theta} \ \Rightarrow \ M_{\theta}(\infty) = 0 \ P$-a.s. \newline
$\underline{ \mathbf{ Case \ B } \ (p \in (0,1)) \text{, inhomogeneous subcritical branching} }$: 
\[
\text{Let} \ \hat{b} = \dfrac{1}{1-p} \text{,} \quad \hat{c} = 
\Big( \dfrac{\beta (1-p)^2}{p} \Big)^{\hat{b}} \text{ as in \eqref{pzero_one}.}   
\]
\noindent Consider $ \theta = (\theta^+, \theta^-)$, where $\theta^-(\cdot) \equiv 1$, and for a given $c > 0$,
\[
\theta^+(s) := \dfrac{c}{\lambda (1-p)} \dfrac{s^{\hat{b}-1}}{(\log (s +2))^{\hat{b}}} \text{ , } s \geq 0 \text{.}
\]
\noindent Then \newline
i) $c < \hat{c} \ \Rightarrow \ M_{\theta}$ is UI and $M_{\theta}(\infty)>0 \ P$-a.s. \newline
ii) $c > \hat{c} \ \Rightarrow \ M_{\theta}(\infty) = 0 \ P$-a.s. \newline
$\underline{ \mathbf{ Case \ C } \ (p = 1) \text{, inhomogeneous near-critical branching} }$:  \newline
Consider $ \theta = (\theta^+, \theta^-)$, where $\theta^-(\cdot) \equiv 1$, and for a given $ \alpha > 0$, 
\[
\theta^+(s) := e^{ \alpha \sqrt{s}} \text{ , } s \geq 0 \text{.}
\] 
\noindent Then \newline
i) $ \alpha < \sqrt{2 \beta} \ \Rightarrow \ M_{\theta}$ is UI and $M_{\theta}(\infty)>0 \ P$-a.s. \newline
ii) $ \alpha > \sqrt{2 \beta} \ \Rightarrow \ M_{\theta}(\infty) = 0 \ P$-a.s.
\end{Theorem}
The importance of this Theorem comes from the fact that if $M_{\theta}$ is 
$P$-uniformly integrable and $M_{\theta}(\infty) > 0 \ P$-a.s. then, as it follows from Lemma \ref{durrett_lem2_}, 
the measures $P$ and $Q_{\theta}$ are equivalent on $ \mathcal{F}_{\infty}$. Since under $\tilde{Q}_{\theta}$ the spine process satisfies 
\[
\frac{\xi_t}{\int_0^t \lambda (\theta^+(s) - \theta^-(s)) \mathrm{d}s} \to 1 \text{ a.s. as } t \to \infty 
\]
it would then follow that under $P$ there is a particle with such asymptotic behaviour 
too. That would give the lower bound on the rightmost particle:
\[
\liminf_{t \to \infty} \frac{R_t}{\int_0^t \lambda ( \theta^+(s) - \theta^-(s)) \mathrm{d}s} \geq 1  \text{,} 
\]
which we can then optimise over suitable $ \theta^+$ and $ \theta^-$.

The upper bound on the rightmost particle needs a slightly different approach, 
which we present in the last subsection. 
\begin{Remark}
\label{thetaremark}
Let us note that the only important feature of $ \theta^+( \cdot)$ in cases $ \mathbf{B}$ 
and $ \mathbf{C}$ is its asymptotic growth. By this we mean that we have freedom 
in defining $ \theta( \cdot)$ as long as we keep
\[
\theta^+(t) \sim \frac{c}{ \lambda (1 - p)} \frac{t^{b-1}}{(\log t)^b} \text{ as } t \to \infty 
\text{ in Case } \mathbf{A}
\]
and 
\[
\log \theta^+(t) \sim \alpha \sqrt{t} \text{ as } t \to \infty \text{ in Case } \mathbf{B} \text{.}
\]
\end{Remark}
\begin{Remark}
\label{iremark}
Parts A ii), B ii) and C ii) of Theorem \ref{M_limit} will not be used in the proof of our main result, Theorem \ref{main}. 
We included them to better illustrate the behaviour of martingales $M_{\theta}$.
\end{Remark}
Recall Lemma \ref{durrett_lem2_}, which says that for events $A \in \mathcal{F}_{\infty}$
\begin{equation}
\label{durrett_lem21}
Q_{\theta} \big( A \big) = \int_A \limsup_{t \to \infty} M_{\theta}(t) \mathrm{d}P + 
Q_{\theta} \big( A \cap \{ \limsup_{t \to \infty} M_{\theta}(t) = \infty \} \big)
\end{equation}
Immediate consequences of this (after taking $A = \Omega$) are:

$\mathbf{1)} \  Q_{\theta}(\limsup_{t \to \infty} M_{\theta}(t) = \infty) = 1 \ \Leftrightarrow \ \limsup_{t \to \infty} M_{\theta}(t) = 0 \ P$-a.s. 
So to prove parts A ii), B ii) and C ii) of Theorem \ref{M_limit} we need to show that 
$ \limsup_{t \to \infty}M_{\theta}(t) = \infty \ Q_{\theta}$-a.s.

$\mathbf{2)} \ Q_{\theta} ( \limsup_{t \to \infty}M_{\theta}(t) < \infty ) = 1 \Leftrightarrow E M_{\theta}(\infty) = 1$ 
in which case $P(M_{\theta}(\infty) > 0) > 0$ and $M_{\theta}$ is $L^1$-convergent w.r.t $P$ as 
it follows from Scheffe's Lemma. Thus $M_{\theta}$ is $P$-uniformly integrable. So to prove the uniform integrability in parts A i), B i) and C i) of 
Theorem \ref{M_limit} we need to show that $ \limsup_{t \to \infty}M_{\theta}(t) < \infty \ Q_{\theta}$-a.s.

The fact that $P(M_{\theta}(\infty) > 0) = 1$ (in parts A i), B i) and C i)) requires additionally 
a certain zero-one law, which we shall give at the end of this subsection. 
\begin{proof}[Proof of Theorem \ref{M_limit}: uniform integrability in A i), B i),  C i)]
We start with proving that for the given values of $ \theta$ in A i), B i) and C i) 
$M_{\theta}$ is UI. As we just said above, it is sufficient to prove that 
\begin{equation}
\label{important}
\limsup_{t \to \infty} M_{\theta}(t) < \infty \ Q_{\theta} \text{-a.s.} 
\end{equation}
for the given paths $\theta$. We have already seen how to do this using the spine 
decomposition in Section 5. Just as before it is sufficient for us to check that
\begin{equation}
\label{eq_3a}
\limsup_{t \to \infty} E^{\tilde{Q}_{\theta}}(M_{\theta}(t) \vert \tilde{\mathcal{G}}_{\infty}) = 
\limsup_{t \to \infty} \big( spine(t) + sum(t) \big) < \infty \ \tilde{Q}_{\theta} \text{-a.s.} 
\end{equation}
Let us outline the main steps of proving \eqref{eq_3a} in cases A, B and C.

$\underline{ \mathbf{ Case \ A } \ (p = 0) \text{, homogeneous branching} }$: \newline
We note that under $ \tilde{Q}_{\theta}$, $(\xi_t^+)_{t \geq 0} \stackrel{d}{=} PP(\lambda \theta_0)$ and 
$(\xi^-_t)_{t \geq 0} \stackrel{d}{=} PP(\frac{\lambda}{\theta_0})$. Hence 
\[
\frac{\xi_t^+}{t} \to \lambda \theta_0 \text{ and } \frac{\xi_t^-}{t} \to \frac{\lambda}{\theta_0}
\quad \tilde{Q}_{\theta} \text{-a.s.}
\]
Then using the above convergence results we wish to show that there exist some positive constant $C''$ and a $ \tilde{Q}_{\theta}$-a.s. 
finite time $T'$ such that $ \forall t > T'$
\begin{equation}
\label{eq_078}
spine(t) \leq  e^{- C'' t} \text{.}
\end{equation}
We observe that for any $\epsilon>0$ there exists a $ \tilde{Q}_{\theta}$-a.s. finite time $T_{\epsilon}$ such that $\forall t > T_{\epsilon}$ 
$(1 - \epsilon) \lambda \theta_0 t \leq \xi^+_t \leq (1 + \epsilon) \lambda \theta_0 t $ and 
$(1 - \epsilon) \frac{\lambda}{\theta_0} t \leq \xi^-_t \leq (1 + \epsilon) \frac{\lambda}{\theta_0} t$. Thus for $t > T_{\epsilon}$
\begin{align*}
spine(t) &\leq \exp \Big( \lambda (1 + \epsilon) \theta_0 \log
\theta_0 t + \lambda (1 - \epsilon) \frac{1}{\theta_0} \log \big( \frac{1}{\theta_0} \big) t  + \lambda \big( 2 - \theta_0 - \frac{1}{\theta_0} \big) t - \beta t \Big) \\
&= \exp \Big( \big( \lambda \big[ g(\theta_0) + \epsilon 
\big(\theta_0 + \dfrac{1}{\theta_0}\big) \log \theta_0   
\big] - \beta \big)t \Big) \text{,}
\end{align*} 
where 
\begin{equation}
\label{g_fun}
g(\theta) = \big( \theta - \frac{1}{\theta} \big) \log \theta - \big( \theta + \frac{1}{\theta} \big) + 2 \text{, } \theta \in [1, \infty)
\end{equation}
is an increasing function such that $g(\hat{\theta}) = \frac{\beta}{\lambda}$ (see the definition of $\hat{\theta}$). Then since $\theta_0 < \hat{\theta}$ it 
follows that for $\epsilon$ small enough
\[
\lambda \Big( g(\theta_0) + \epsilon \big( \theta_0 + \frac{1}{\theta_0} \big) \log \theta_0 \Big) - \beta < 0 \text{.}
\]
We thus take $T' = T_{\epsilon}$ for such an $\epsilon$ and 
$C'' = - \lambda \Big( g(\theta_0) + \epsilon \big( \theta_0 + \frac{1}{\theta_0} \big) \log \theta_0 \Big) - \beta$ 
to obtain \eqref{eq_078}.

Then we have 
\begin{align*}
sum(t) & = \sum_{u < node_t(\xi)} spine(S_u) \\
& \leq \Big( \sum_{u < node_t(\xi), \ S_u \leq T'} spine(S_u) \Big) + 
\Big( \sum_{u < node_t(\xi), \ S_u > T'} e^{- C'' S_u} \Big) \text{,}
\end{align*}
where the first sum, call it $C_1$, is $ \tilde{Q}_{\theta}$-a.s. bounded 
since it only counts births up to time $T'$. Thus 
\begin{equation}
\label{eq_www3}
sum(t) \leq \ C_1 + \sum_{n=1}^{\infty} e^{ -C '' S_n } \text{,}
\end{equation}
where $S_n$ is the time of the $n^{th}$ birth on the spine.

The birth process along the spine $(n_t)_{t \in [0, \infty)}$ is a Poisson process with rate $2 \beta $. 
Therefore $t^{-1}n_t \to 2 \beta \ \tilde{Q}_{\theta} $-a.s. as $t \to \infty$ and hence
$n^{-1}S_n \to (2 \beta)^{-1} \ \tilde{Q}_{\theta}$-a.s. as $ n \to \infty$. 
So for some $ \tilde{Q}_{\theta}$-a.s. finite positive random variable $C_2$ we have 
$S_n \geq C_2 n \quad \forall n $. Then substituting this into \eqref{eq_www3} we get 
\[
sum(t) \leq C_1 + \sum_{n = 1}^{ \infty} e^{-C'' C_2 n} < \infty \ \tilde{Q}_{\theta} \text{-a.s.,} 
\]
which gives \eqref{eq_3a}.

$\underline{ \mathbf{ Case \ B } \ (p \in (0,1)) \text{, inhomogeneous subcritical branching} }$: \newline
From Proposition \ref{prop_ch_meas} under $ \tilde{Q}_{\theta}$ the process $(\xi^+_t)_{t \in [0, \infty)}$ is a time-inhomogeneous 
Poisson process with jump rate $ \lambda \theta^+(t)$ and $(\xi^-_t)_{t \in [0, \infty)}$ is a Poisson process of rate $ \lambda$. 
Then from Propositions \ref{poisson_c} and \ref{poisson_d} we find that, $\tilde{Q}_{\theta}$-a.s.,
\[
\frac{\xi^+_t}{\int_0^t \lambda \theta^+(s) \mathrm{d}s} \to 1 \text{ ,  } \ 
\frac{\xi^-_t}{\lambda t} \to 1 \text{ , } \ 
\frac{\int_0^t \log \theta^+(s) \mathrm{d} \xi_s^+}{\int_0^t \log \theta^+(s)
\ \lambda \theta^+(s) \mathrm{d}s} \to 1 \text{.}
\]
It can then be checked in a similar way as before that there exist some $ \tilde{Q}_{\theta}$-a.s. finite positive 
random variables $C'$, $C''$ and  $T'$ such that, $ \forall t > T'$, 
\[
spine(t) \leq C' \exp \Big( - C'' \int_0^t \frac{s^{\hat{b}p}}{(\log(s+2))^{\hat{b}p}} \mathrm{d}s \Big) \text{.}
\]
For the sum term of the spine decomposition we have when $t > T'$
\[
sum(t) \leq \sum_{\buildrel  u < node_t(\xi), \over{S_u \leq T'}} spine(S_u) 
+ \sum_{\buildrel u < node_t(\xi), \over{S_u > T'}} C' \exp \Big( - C'' \int_0^{S_u} \frac{s^{\hat{b}p}}{(\log(s+2))^{\hat{b}p}} \mathrm{d}s \Big)
\]
The first sum is a $ \tilde{Q}_{\theta}$-a.s. finite random variable which doesn't depend on $t$, and which we call $C_1$. Then
\begin{equation}
\label{eq_www4}
sum(t) \leq \ C_1 + C' \sum_{n=1}^{\infty} \exp \Big( -C'' \int_0^{S_n} \frac{s^{\hat{b}p}}{(\log(s+2))^{\hat{b}p}} \mathrm{d}s \Big) \text{,}
\end{equation}
where $S_n$ is the time of the $n^{th}$ birth on the spine.

The birth process along the spine $(n_t)_{t \in [0, \infty)}$ conditional on the path of the spine 
is time-inhomogeneous Poisson process (or Cox process) with jump rate $2 \beta |\xi_t|^p$ at 
time $t$. Thus, we find 
\[
n_t \sim 2 \beta \int_0^t |\xi_s|^p \mathrm{d}s \sim 2 \beta \Big( \frac{c}{\hat{b}(1-p)} \Big)^p 
\int_0^t  \frac{s^{\hat{b}p}}{(\log(s+2))^{\hat{b}p}} \mathrm{d}s 
\quad \tilde{Q}_{\theta} \text{-a.s. as } t \to \infty \text{.}
\]
So for some $ \tilde{Q}_{\theta}$-a.s. finite positive random variable $C_2$ we have  
\[
\int_0^{S_n} \frac{s^{\hat{b}p}}{(\log(s+2))^{\hat{b}p}} \mathrm{d}s \geq C_2 n \quad \forall n \text{.}
\]
Then substituting this into \eqref{eq_www4} we verify that \eqref{eq_3a} again holds.

$\underline{ \mathbf{ Case \ C } \ (p = 1) \text{, inhomogeneous near-critical branching} }$: \newline
As in the previous case, under $ \tilde{Q}_{\theta}$ the process $(\xi^+_t)_{t \in [0, \infty)}$ is a time-inhomogeneous 
Poisson process with jump rate $ \lambda \theta^+(t)$ and $(\xi^-_t)_{t \in [0, \infty)}$ is a Poisson process of rate $ \lambda$. 
Then $\tilde{Q}_{\theta}$-a.s. we have 
\[
\frac{\xi^+_t}{\int_0^t \lambda \theta^+(s) \mathrm{d}s} \to 1 \text{ ,  } \ 
\frac{\xi^-_t}{\lambda t} \to 1 \text{ , } \ 
\frac{\int_0^t \log \theta^+(s) \mathrm{d} \xi_s^+}{\int_0^t \log \theta^+(s)
\ \lambda \theta^+(s) \mathrm{d}s} \to 1 \text{.}
\]
One can check that there exist some $ \tilde{Q}_{\theta}$-a.s. finite positive random variables $C'$, $C''$ and $T'$ such that, $ \forall t > T'$,
\[
spine(t) \leq C' \exp \Big( - C'' \int_0^t \sqrt{s} e^{\alpha \sqrt{s}} \mathrm{d}s \Big) \text{.}
\]
Then for $t > T'$
\begin{equation}
\label{eq_www5}
sum(t)  \leq C_1 + C' \sum_{n=1}^{\infty} \exp \Big( -C'' \int_0^{S_n}  \sqrt{s} e^{\alpha \sqrt{s}} \mathrm{d}s \Big) \text{,}
\end{equation}
where $C_1 < \infty$ and $S_n$ is the time of the $n^{th}$ birth on the spine.
The birth process along the spine $(n_t)_{t \in [0, \infty)}$ then satisfies
\[
n_t \sim \int_0^t 2 \beta |\xi_s| \mathrm{d}s 
\sim \frac{4 \beta \lambda}{\alpha} \int_0^t \sqrt{s} e^{\alpha \sqrt{s}} \mathrm{d}s
\quad \tilde{Q}_{\theta} \text{-a.s. as } t \to \infty \text{.}
\]
So for some $ \tilde{Q}_{\theta}$-a.s. finite positive random variable $C_2$ we have  
\[
\int_0^{S_n} \sqrt{s} e^{\alpha \sqrt{s}} \mathrm{d}s \geq C_2 n \quad \forall n \text{.}
\]
Then substituting this into \eqref{eq_www5} we again find that \eqref{eq_3a} holds. 

Thus we have completed the proof of uniform integrability 
and the fact that $P(M_{\theta}(\infty) > 0) > 0$ in Theorem \ref{M_limit}. 
\end{proof}
\begin{proof}[Proof of Theorem \ref{M_limit}: parts A ii), B ii), C ii)]
Since one of the particles at time $t$ is the spine, we have
\begin{align*}
M_{\theta}(t) \ \geq \ \exp \Big( & \int_0^t \log(\theta^+(s)) \mathrm{d}
\xi^+_s \ + \ \int_0^t \log(\theta^-(s)) \mathrm{d} \xi^-_s \\ + & \ \lambda
\int_0^t (2 - \theta^+(s) - \theta^-(s)) \mathrm{d}s \ - \ \beta \int_0^t
\vert \xi_s \vert^p \mathrm{d}s \Big) \ = \ spine(t) \text{.}
\end{align*}
For the paths $ \theta$ in parts ii) of Theorem \ref{M_limit} one can check
(following the same analysis as in the proof of parts i) of the
Theorem) that $spine(t) \to \infty \ $
$\tilde{Q}_{\theta}$-a.s. Thus 
\[
\limsup_{t \to \infty} M_{\theta}(t) = \infty \ \tilde{Q}_{\theta} \text{-a.s.} 
\]
and so also
$Q_{\theta}$-a.s. Recalling \eqref{durrett_lem21} we see that
$M_{\theta}(\infty) = 0 \ P$-a.s. for the proposed choices of 
$ \theta$.
\end{proof}
It remains to show that in Theorem \ref{M_limit} $P(M_{\theta}(\infty) > 0) = 1$ when $M_{\theta}$ is UI. 
The following 0-1 law will do the job. 
\begin{Lemma}
\label{zero_one_law}
Let $q : \mathbb{Z} \to [0,1]$ be such that $M_t := \prod_{u \in N_t}
q(X_u(t))$ is a $P$-martingale (usually referred to as a product martingale). Then $q(x) \equiv q \in \{ 0,1 \}$.
\end{Lemma}
\begin{proof}[Proof of Lemma \ref{zero_one_law}]
Since $M_t$ is a martingale and one of the particles alive at time $t$
is the spine we have
\[
q(x) = E^x M_t = \tilde{E}^x M_t \leq \tilde{E}^x q(\xi_t) \text{.}
\]
So $q(\xi_t)$ is a positive $\tilde{P}$-submartingale. Since it is bounded it converges
$\tilde{P}$-a.s. to some limit $q_{\infty}$. We also know that under $ \tilde{P}$, $ (\xi_t)_{t \geq 0}$ 
is a continuous-time random walk. Recurrence of $(\xi_t)_{t \geq 0}$ implies that $q_{\infty} \equiv q(0)$ and that $q(x)$ is constant in $x$. 

Now suppose for contradiction that $q(0) \in (0,1)$. Then 
\[
M_t = \prod_{u \in N_t} q(X_u(t)) = q(0)^{|N_t|} \to 0
\]
because $|N_t| \to \infty$. Since $M$ is bounded it is uniformly integrable, so $q(0) = E M_{\infty} = 0$, which is a contradiction. 
So $q(0) \notin (0, 1)$ and thus $q(0) \in \{ 0,1 \}$. 
\end{proof}
\begin{proof}[Proof of Theorem \ref{M_limit}: positivity of limits in A i), B i), C i)]
We apply Lemma \ref{zero_one_law} to $q(x) = P^x (M_{\theta}(\infty) = 0)$. By the tower propery of conditional 
expectations and the branching Markov property we have 
\[
q(x) = E^x \Big( P^x \big(M_{\theta}(\infty) = 0 \big\vert \mathcal{F}_t
\big) \Big) = E^x \Big( \prod_{u \in N_t} q \big( X_u(t) \big) \Big)
\] 
whence $\prod_{u \in N_t} q(X_u(t))$ is a $P$-martingale. Also $E (M_{\theta}(\infty)) = M_{\theta}(0) = 1 > 0$. Therefore
$P(M_{\theta}(\infty) = 0) \neq 1$. So by Lemma \ref{zero_one_law} $P(M_{\theta}(\infty) = 0) = 0 $.
\end{proof}
\subsection{Lower bound on the rightmost particle} 
\begin{Proposition}
\label{lowerbound}
Let $\hat{\theta}$, $\hat{b}$ and $\hat{c}$ be as defined in Theorem \ref{main}. Then

$\underline{\mathbf{Case \ A } \ (p = 0)}$:
\[
\liminf_{t \to \infty} \frac{R_t}{t} \geq \lambda(\hat{\theta} - \frac{1}{\hat{\theta}}) \ P \text{-a.s.}
\]

$\underline{\mathbf{Case \ B } \ (p \in (0,1))}$:
\[
\liminf_{t \to \infty} \Big( \dfrac{\log t}{t} \Big)^{\hat{b}} R_t
\geq \hat{c} \ P \text{-a.s.}
\]

$\underline{\mathbf{Case \ C } \ (p = 1)}$:
\[
\liminf_{t \to \infty} \frac{\log R_t}{\sqrt{t}} \geq \sqrt{2 \beta} \ P \text{-a.s.}
\]
\end{Proposition}
\begin{proof}
$ $ 

\noindent $\underline{\mathbf{Case \ A } \ (p = 0)}$:

We consider $\theta = (\theta^+, \theta^-)$, where $\theta^+(\cdot) \equiv \theta_0$, 
$\theta^-(\cdot) \equiv \frac{1}{\theta_0}$ and $ \theta_0 < \hat{\theta}$. Take the event
\[
B_{\theta_0} := \Big\{ \exists \text{ infinite line of descent } u : \liminf_{t \to \infty} \frac{X_u(t)}{t}
= \lambda(\theta_0 - \frac{1}{\theta_0}) \Big\} \in \mathcal{F}_{\infty} \text{.}
\]
We know that $\tilde{Q}_{\theta} (\lim_{t \to \infty} \frac{\xi_t}{t} =  \lambda(\theta_0 - \frac{1}{\theta_0})) = 1$. 
Hence $Q_{\theta} (B_{\theta_0}) = \tilde{Q}_{\theta} (B_{\theta_0}) = 1$. Since $Q_{\theta}$ and $P$ 
are equivalent it follows that $P (B_{\theta_0}) = 1$. Thus 
$ P \big( \liminf_{t \to \infty} t^{-1}R_t \geq  \lambda(\theta_0 - \theta_0^{-1}) \big) = 1 $. 
Taking the limit $\theta_0 \nearrow \hat{\theta}$ we get 
\[
P \Big( \liminf_{t \to \infty} \frac{R_t}{t} \geq \lambda \big( \hat{\theta} - \frac{1}{\hat{\theta}} \big) \Big) = 1 \text{.}
\]
$\underline{\mathbf{Case \ B } \ (p \in (0,1))}$:

Consider $\theta = (\theta^+, \theta^-)$, where $\theta^-(\cdot) \equiv 1$, 
$\theta^+(s) = \dfrac{c}{\lambda(1-p)} \dfrac{s^{\hat{b}-1}}{(\log (s+2))^{\hat{b}}}$ and $ c < \hat{c}$.
Take the event
\[
B_c := \Big\{ \exists \text{ infinite line of descent } u : \liminf_{t \to \infty} \Big( \dfrac{\log
t}{t}\Big)^{\hat{b}}X_u(t) = c \Big\} \text{.}
\]
Same argument as above gives that $P (B_{c}) = 1$ and hence 
$P \Big( \liminf_{t \to \infty}\Big( t^{-1} \log t \Big)^{\hat{b}} R_t \geq c \Big) = 1$ for all $c < \hat{c}$ . 
Letting $c \nearrow \hat{c}$ proves the result. \newline
$\underline{\mathbf{Case \ C } \ (p = 1)}$:

Consider $\theta = (\theta^+, \theta^-)$, where $\theta^-(\cdot) \equiv 1$, 
$\theta^+(s) = e^{\alpha \sqrt{s}}$ and $ \alpha < \sqrt{2 \beta}$.
Take the event
\[
B_{\alpha} := \Big\{ \exists \text{ infinite line of descent } u : \liminf_{t \to \infty} \frac{\log X_u(t)}{\sqrt{t}}
= \sqrt{2 \beta} \Big\} \text{.}
\]
Again, the same argument as above gives $P (B_{\alpha}) = 1$ and hence for all $ \alpha < \sqrt{2 \beta}$ we find that 
$P \Big( \liminf_{t \to \infty} t^{-1/2}\log R_t \geq \alpha \Big) = 1$. Letting 
$\alpha \nearrow \sqrt{2 \beta}$ proves the result. 
\end{proof}
\subsection{Upper bound on the rightmost particle} 
To complete the proof of Theorem \ref{main} and hence the whole section we need to prove 
the following proposition.
\begin{Proposition}
\label{upperbound}
Let $\hat{\theta}$, $\hat{b}$ and $\hat{c}$ be as defined in Theorem \ref{main}. Then for different 
values of $p$ we have the following.

$\underline{\mathbf{Case \ A } \ (p = 0)}$:
\[
\limsup_{t \to \infty} \frac{R_t}{t} \leq \lambda(\hat{\theta} - \frac{1}{\hat{\theta}}) \ P \text{-a.s.}
\]

$\underline{\mathbf{Case \ B } \ (p \in (0,1))}$:
\[
\limsup_{t \to \infty} \Big( \dfrac{\log t}{t} \Big)^{\hat{b}} R_t
\leq \hat{c} \ P \text{-a.s.}
\]

$\underline{\mathbf{Case \ C } \ (p = 1)}$:
\[
\limsup_{t \to \infty} \frac{\log R_t}{\sqrt{t}} \leq \sqrt{2 \beta} \ P \text{-a.s.}
\]
\end{Proposition}
To prove Proposition \ref{upperbound} we shall assume for contradiction that 
it is false. Then we shall show that under such assumption certain additive 
$P$-martingales will diverge to $ \infty$ contradicting the Martingale Convergence Theorem.

We start by proving the following 0-1 law.
\begin{Lemma}
\label{zero_one2}
Let $g:[0, \infty) \to \mathbb{R}$ be increasing, $f:[0, \infty) \to [0, \infty)$ 
be such that $ \forall s \geq 0 \ \frac{f(t)}{f(s+t)} \to 1$ as $t \to \infty$ and $a > 0$. Then 
\[
P \Big( \limsup_{t \to \infty} \dfrac{g(R_t)}{f(t)} \leq a \Big) \in \{ 0, 1 \} \text{.}
\]
\end{Lemma}
\begin{proof}
We consider
\[
q(x) = P^x \Big( \limsup_{t \to \infty} \dfrac{g(R_t)}{f(t)} \leq a \Big) \text{.}
\]
Then, it is easy to see that  
\begin{align*}
q(x) &= E^x \Big( P^x \big( \limsup_{t \to \infty} \frac{g(R_{t+s})}{f(t+s)} \leq a \big\vert \mathcal{F}_s \big) \Big) \\
&=  E^x \Big( P^x \big( \limsup_{t \to \infty} \frac{g( \max_{u \in N_s} R^u_t)}{f(t+s)} \leq a \big\vert \mathcal{F}_s \big) \Big) \\
&=  E^x \Big( P^x \big( \max_{u \in N_s} \big\{ \limsup_{t \to \infty} \frac{g(R^u_t)}{f(t+s)} \big\}
\leq a \big\vert \mathcal{F}_s \big) \Big) \\
&= E^x \Big( \prod_{u \in N_s} P^{X_u(s)} \big( \limsup_{t \to \infty} \frac{g(R_t)}{f(t + s)} \leq a 
\big) \Big) \\
&= E^x \Big( \prod_{u \in N_s} P^{X_u(s)} \big( \limsup_{t \to \infty} \frac{g(R_t)}{f(t)} \leq a 
\big) \Big) \\
&= E^x \Big( \prod_{u \in N_s} q \big( X_u(s) \big) \Big) \text{,}
\end{align*}
where $(R^u_t)_{t \geq 0}$ is the position of the rightmost particle of a subtree started from $X_u(s)$.

Thus $\prod_{u \in N_t} q(X_u(t))$ is a martingale. Applying Lemma \ref{zero_one_law} to $q(\cdot)$ we obtain the required result. 
\end{proof}
\begin{proof}[Proof of Proposition \ref{upperbound}]
The first step of the proof is slightly different for cases A, B and C, so we 
do it for the three cases separately.

$\underline{\mathbf{Case \ A } \ (p = 0)}$

Let us suppose for contradiction that $ \exists \theta_0 > \hat{\theta}$ such that 
\begin{equation}
\label{eq_x}
P \Big( \limsup_{t \to \infty} \frac{R_t}{t} > \lambda(\theta_0 - \frac{1}{\theta_0}) \Big) = 1 \text{.}
\end{equation}
Choose any $\theta_A \in (\hat{\theta}, \theta_0)$ and take $ \theta = (\theta^+, \theta^-)$, where 
$\theta^+(\cdot) \equiv \theta_A$, $\theta^-(\cdot) = \dfrac{1}{\theta_A}$.
Let 
\[
f_A(s) := \lambda(\theta_A - \frac{1}{\theta_A}) s \text{, } \quad s \geq 0 \text{.}
\]

$\underline{\mathbf{Case \ B } \ (p \in (0,1))}$

Let us suppose for contradiction that $ \exists c_0 > \hat{c}$ such that 
\begin{equation}
\label{eq_y}
P \Big( \limsup_{t \to \infty} \big( \frac{\log t}{t} \big)^{\hat{b}} R_t > c_0 \Big) = 1 \text{.}
\end{equation}
Choose any $c_1 \in (\hat{c}, c_0)$ and take $ \theta = (\theta^+, \theta^-)$, where 
$\theta^+(s) = \theta_B(s)$, $\theta^-(s) = \dfrac{1}{\theta_B(s)}$ and 
\[
\theta_B(s) = \frac{c_1}{\lambda (1-p)} \frac{s^{\hat{b}-1}}{(\log (s+2))^{\hat{b}}} \text{, } \quad s \geq 0 \text{.}
\] 
Let 
\[
f_B(s) := c_1 \Big( \frac{s}{ \log (s+2)} \Big)^{\hat{b}} \text{, } \quad s \geq 0 \text{.}
\]

$\underline{\mathbf{Case \ C } \ (p = 1)}$

Let us suppose for contradiction that $ \exists \alpha_0 > \sqrt{2 \beta}$ such that 
\begin{equation}
\label{eq_z}
P \Big( \limsup_{t \to \infty} \frac{\log R_t}{\sqrt{t}} > \alpha_0 \Big) = 1 \text{.}
\end{equation}
Choose any $\alpha_1 \in (\sqrt{2 \beta}, \alpha_0)$ and take $ \theta = (\theta^+, \theta^-)$, where 
$\theta^+(s) = \theta_C(s)$, $\theta^-(s) = \dfrac{1}{\theta_C(s)}$ and 
\[
\theta_C(s) = \frac{1}{\sqrt{s+1}} e^{\alpha_1 \sqrt{s}} \text{, } \quad s \geq 0 \text{.}
\] 
Let 
\[
f_C(s) := e^{\alpha_1 \sqrt{s}} \text{, } \quad s \geq 0 \text{.}
\]
The next step in the proof is the same in all cases. 

Let us write $f$ to denote $f_A$, $f_B$ and $f_C$. We define $D(f)$ to be the space-time region 
bounded above by the curve $y = f(t)$ and below by the curve $y = - f(t)$. 

Under $P$ the spine process $(\xi_t)_{t \geq 0}$ is a continuous-time random walk 
and so \newline $\dfrac{\vert \xi_t \vert}{t} \to 0 \ P$-a.s. as $t \to \infty$. Hence there exists 
an a.s. finite random time $T' < \infty$ such that $\xi_t \in D(f)$ for all $t > T'$.
 
Since $(\xi_t)_{t \geq 0}$ is recurrent it will spend an infinite amount of time at 
position $y = 1$. During this time it will be giving birth to offspring at rate
$\beta$. This assures us of the existence of an infinite sequence $\{ T_n \}_{n \in \mathbb{N}}$ 
of birth times along the path of the spine when it stays at $y = 1$ with 
$0 \leq T' \leq T_1 < T_2 < ...$ and $T_n \nearrow \infty$. 

Denote by $u_n$ the label of the particle born at time $T_n$, which does not continue the spine. 
Then each particle $u_n$ gives rise to an independent copy of the Branching
random walk under $P$ started from $\xi_{T_n}$ at time $T_n$. Almost surely, by assumptions 
\eqref{eq_x}, \eqref{eq_y} and \eqref{eq_z}, each $u_n$ has some descendant that leaves the space-time region $D(f)$. 

Let $\{ v_n \}_{n \in \mathbb{N}}$ be the subsequence of $ \{ u_n \}_{n \in \mathbb{N}} $ of those particles
whose first descendent leaving $D(f)$ does this by crossing the upper boundary $y = f(t)$. 
Since the breeding potential is symmetric and the particles 
$u_n$ are born in the upper half-plane, there is at least probability 
$\frac{1}{2}$ that the first descendant of $u_n$ to leave $D(f)$ 
does this by crossing the positive boundary curve. Therefore $P$-a.s. the 
sequence $ \{ v_n\}_{n \in \mathbb{N}} $ is infinite. 

Let $w_n$ be the decsendent of $v_n$, which exits $D(f)$ first and let $J_n$ 
be the time when this occurs. That is, 
\[
J_n = \inf \big\{t : X_{w_n}(t) \geq f(t) \big\} \text{.}
\] 
Note that the path of particle $w_n$ satisfies 
\[
|X_{w_n}(s)| < f(s) \quad \forall s \in [T', J_n) \text{.}
\]
 Clearly $J_n \to \infty$ as $n \to \infty$. To obtain a contradiction we shall show that the additive martingale $M_{\theta}$ 
fails to converge along the sequence of times $ \{ J_n \}_{n \geq 1}$, where $ \theta$ was defined above differently 
for cases A, B and C. Thus for the last bit of the proof we have to look at cases A, B and C separately again. \newline
$\underline{\mathbf{Case \ A } \ (p = 0)}$
\begin{align*}
M_{\theta}(J_n) &= \sum_{u \in N_{J_n}} \exp \Big\{ \int_0^{J_n} \log \theta_A \mathrm{d}X_u^+(s) 
+ \int_0^{J_n} \log \big( \frac{1}{\theta_A} \big) \mathrm{d}X_u^-(s) \\
& \qquad + \lambda \int_0^{J_n} \big(2 - \theta_A - \frac{1}{\theta_A} \big) \mathrm{d}s 
- \beta \int_0^{J_n} 1 \mathrm{d}s \Big\} \\ 
&\geq \exp \Big\{ \int_0^{J_n} \log \theta_A \mathrm{d}X^+_{w_n}(s) + \int_0^{J_n} \log
\big(\frac{1}{\theta_A}\big) \mathrm{d}X^-_{w_n}(s) \\ 
& \qquad + \lambda \int_0^{J_n} \big(2 - \theta_A - \frac{1}{\theta_A}\big) \mathrm{d}s - 
\beta \int_0^{J_n} 1 \mathrm{d}s \Big\} \\
&= \exp \Big\{ \log \theta_A X^+_{w_n}(J_n) - \log \theta_A X^-_{w_n}(J_n) 
+ \lambda \big(2 - \theta_A - \frac{1}{\theta_A} \big) J_n - \beta J_n \Big\} \\
&= \exp \Big\{ \log \theta_A X_{w_n}(J_n) + \lambda \big(2 -
\theta_A - \frac{1}{\theta_A} \big) J_n - \beta J_n  \Big\} \\
&\geq \exp \Big\{ a_1 J_n \log \theta_A + \lambda \big(2 - \theta_A - 
\frac{1}{\theta_A} \big) J_n - \beta J_n \Big\} \\
&= \exp \Big\{ \Big( \lambda\big((\theta_A - \frac{1}{\theta_A}\big)\log 
\theta_A + \lambda \big(2 - \theta_A - \frac{1}{\theta_A} \big) - \beta \Big) J_n \Big\} \\
&= \exp \Big\{ \Big( \lambda g(\theta_A) - \beta \Big)J_n \Big\} \text{,}
\end{align*}
where $g(\cdot)$ is the same as in \eqref{g_fun}. Then since $g(\cdot)$ is increasing, $ \theta_A > \hat{\theta}$ and 
$g(\hat{\theta}) = \frac{\beta}{\lambda}$ it follows that 
\[
\lambda g(\theta_A) - \beta > 0
\]
and thus $M_{\theta}(J_n) \to \infty$ as $n \to \infty$, which is a contradiction.
Therefore assumption \eqref{eq_x} is wrong and we must have that $ \forall \theta_0 > \hat{\theta}$ 
\[
P \Big( \limsup_{t \to \infty} \frac{R_t}{t} > \lambda(\theta_0 - \frac{1}{\theta_0}) \Big) \neq 1 \text{.}
\]
It follows from Lemma \ref{zero_one2} that $ \forall \theta_0 > \hat{\theta} \ P \Big( \limsup_{t \to \infty} \frac{R_t}{t} > \lambda(\theta_0 - \frac{1}{\theta_0}) \Big) = 0$. 
Hence $P \Big( \limsup_{t \to \infty} \frac{R_t}{t} \leq \lambda(\theta_0 - \frac{1}{\theta_0}) \Big) = 1$ and after 
letting $\theta_0 \searrow \hat{\theta}$ we get 
\[
P \Big( \limsup_{t \to \infty} \frac{R_t}{t} \leq  \lambda(\hat{\theta} - \frac{1}{\hat{\theta}}) \Big) = 1 \text{.}
\]
$\underline{\mathbf{Case \ B } \ (p \in (0,1))}$
\begin{align*}
M_{\theta}(J_n) &= \sum_{u \in N_{J_n}} \exp \Big\{ \int_0^{J_n} \log \theta_B(s)
\mathrm{d}X_u^+(s) + \int_0^{J_n} \log \big( \frac{1}{\theta_B(s)} \big) \mathrm{d}X_u^-(s) \\
& \qquad + \lambda \int_0^{J_n} \big(2 - \theta_B(s) - \frac{1}{\theta_B(s)} \big) \mathrm{d}s 
- \beta \int_0^{J_n} \vert X_u(s) \vert^p \mathrm{d}s \Big\} \\ 
&\geq \exp \Big\{ \int_0^{J_n} \log \theta_B(s) \mathrm{d}X_{w_n}^+(s) + 
\int_0^{J_n} \log \big( \frac{1}{\theta_B(s)} \big) \mathrm{d}X_{w_n}^-(s) \\
& \qquad + \lambda \int_0^{J_n} \big(2 - \theta_B(s) - \frac{1}{\theta_B(s)} \big) \mathrm{d}s 
- \beta \int_0^{J_n} \vert X_{w_n}(s) \vert^p \mathrm{d}s \Big\} \text{.}
\end{align*}
Applying the integration by parts formula from Proposition \ref{poisson_b} we get 
\begin{align*}
&\exp \Big\{ \log \theta_B(J_n) X^+_{w_n}(J_n) - \int_0^{J_n} \frac{\theta_B'(s)}{\theta_B(s)} X^+_{w_n}(s) \mathrm{d}s \\
& \qquad - \log \theta_B(J_n) X^-_{w_n}(J_n) + \int_0^{J_n} \frac{\theta_B'(s)}{\theta_B(s)} X^-_{w_n}(s) \mathrm{d}s \\
& \qquad +  \lambda \int_0^{J_n} \big( 2 - \theta_B(s) - \frac{1}{\theta_B(s)} \big) 
\mathrm{d}s - \beta \int_0^{J_n} \vert X_{w_n}(s) \vert^p \mathrm{d}s \Big\} \\
= &\exp \Big\{ \log \theta_B(J_n) X_{w_n}(J_n) - \int_0^{J_n} \frac{\theta_B'(s)}{\theta_B(s)} X_{w_n}(s) \mathrm{d}s \\
& \qquad + \lambda \int_0^{J_n} \big( 2 - \theta_B(s) - \frac{1}{\theta_B(s)} \big) \mathrm{d}s - \beta \int_0^{J_n} 
\vert X_{w_n}(s) \vert^p \mathrm{d}s \Big\} \\
\geq &C \exp \Big\{ \log \theta_B(J_n) f_B(J_n) - \int_0^{J_n} \frac{\theta_B'(s)}{\theta_B(s)} f_B(s) \mathrm{d}s \\
& \qquad + \lambda \int_0^{J_n} \big( 2 - \theta_B(s) - \frac{1}{\theta_B(s)} \big) \mathrm{d}s - \beta \int_0^{J_n} 
f_B(s)^p \mathrm{d}s \Big\}
\end{align*}
using the facts that $X_{w_n}(J_n) \geq f_B(J_n)$ and $|X_{w_n}(s)| < f_B(s)$ for $s \in [T', J_n)$ 
and where $C$ is some $P$-a.s positive random variable. Now asymptotic properties of 
$ \theta_B(\cdot)$ and $f_B(\cdot)$ give us that for any $ \epsilon > 0$ 
and $n$ large enough the above expression is 
\[
\geq C_{\epsilon} \exp \Big\{ (\hat{b} - 1)c_1 \frac{(J_n)^{\hat{b}}}{(\log J_n)^{\hat{b}-1}} (1 - \epsilon) 
- \beta c_1^p \dfrac{1}{\hat{b}} \dfrac{(J_n)^{\hat{b}}}{(\log J_n)^{\hat{b}-1}} (1 + \epsilon) \Big\}
\]
for some $P$-a.s. positive random variable $C_{\epsilon}$. Then since $c_1 > \hat{c} = \Big( \frac{\beta(1-p)^2}{p} \Big)^{(1-p)^{-1}}$ 
\[
(\hat{b} - 1) c_1 (1 - \epsilon) - \beta c_1^p \frac{1}{\hat{b}} (1 + \epsilon) = 
c_1^p (\hat{b} - 1)(1 - \epsilon) \Big( c_1^{1 -p} - \hat{c}^{1-p} \frac{1 + \epsilon}{1 - \epsilon} \Big) > 0
\]
for $\epsilon$ small enough. Thus $M_{\theta}(J_n) \to \infty$ as $n \to \infty$, which is a contradiction.
Therefore assumption \eqref{eq_y} is wrong and we must have that $ \forall c_0 > \hat{c}$ 
\[
P \Big( \limsup_{t \to \infty} \big( \dfrac{\log t}{t} \big)^{\hat{b}} R_t > c_0 \Big) \neq 1 \text{.}
\]
It follows from Lemma \ref{zero_one2} that $ \forall c_0 > \hat{c}$ 
\[
P \Big( \limsup_{t \to \infty} \big( \dfrac{\log t}{t} \big)^{\hat{b}} R_t \leq c_0 \Big) = 1
\]
Hence taking the limit $c_0 \searrow \hat{c}$ proves Proposition \ref{upperbound} in Case B. \newline
$\underline{\mathbf{Case \ C } \ (p = 1)}$

Essentially the same argument as in Case B gives that for any $ \epsilon > 0$ 
and $n$ large enough
\[
M_{\theta}(J_n) \geq C_{\epsilon} \exp \Big\{ (1 - \epsilon) \alpha_1 \sqrt{J_n} e^{\alpha_1 \sqrt{J_n}} 
- (1 + \epsilon) \frac{2 \beta}{\alpha_1} \sqrt{J_n} e^{\alpha_1 \sqrt{J_n}} \Big\}
\]
for some $C_{\epsilon} > 0 \ P$-a.s. Then since $ \alpha_1 > \sqrt{2 \beta}$ 
\[
(1 - \epsilon) \alpha_1 - (1 + \epsilon) \frac{2 \beta}{\alpha_1} > 0
\]
for $\epsilon $ chosen sufficiently small. Therefore $M_{\theta}(J_n) \to \infty$, which is a contradiction. 
Hence $ \forall \alpha_0 > \sqrt{2 \beta}$
\[
P \Big( \limsup_{t \to \infty} \dfrac{\log R_t}{\sqrt{t}} \leq \alpha_0 \Big) = 1
\]
and therefore 
\[
P \Big( \limsup_{t \to \infty} \dfrac{\log R_t}{\sqrt{t}} \leq \sqrt{2 \beta} \Big) = 1 \text{.}
\]
This finishes the proof of Proposition \ref{upperbound} and also Theorem \ref{main}
\end{proof}
\bibliographystyle{acm}

\def\cprime{$'$}

\end{document}